\documentclass[11pt]{article}

\usepackage{amsfonts}
\usepackage{amsmath,amssymb}
\usepackage{amscd}
\usepackage{amsthm}
\usepackage{fancyhdr}
\usepackage{color}
\usepackage{hyperref} 
\usepackage{tikz-cd}

\usepackage{txfonts} 

\theoremstyle{plain}
\newtheorem{lem}{Lemma}[section] 
\newtheorem{thm}[lem]{Theorem}

\newtheorem{cor}[lem]{Corollary}

\newtheorem{prop}[lem]{Proposition}

\theoremstyle{definition}
\newtheorem{defn}[lem]{Definition}
\newtheorem{ex}[lem]{Example}

\theoremstyle{remark}
\newtheorem{rk}[lem]{Remark}
\numberwithin{equation}{lem}

\newcommand{\KN}{\mathbb{N}}

\newcommand{\CA}{\mathcal{A}}

\newcommand{\CG}{\mathcal{G}}

\newcommand{\CM}{\mathcal{M}}

\newcommand{\CP}{\mathcal{P}}

\newcommand{\g}{\mathfrak{g}}

\newcommand{\CI}{C^\infty}

\newcommand{\st}{\hspace{.05in}:\hspace{.05in}}

\newcommand{\y}{\hspace{.09in}\text{and}\hspace{.09in}}
\newcommand{\de}{\partial}
\newcommand{\fto}{\rightarrow}
\newcommand{\soutar}{\rightrightarrows}
\newcommand{\bt}{\mathbf{t}}
\newcommand{\bs}{\mathbf{s}}

\title{Principal bundle  groupoids, their gauge group and their nerve.}

\author{Alfonso Garmendia
	\footnote{Departament de Matemàtiques, Universitat Politècnica de Catalunya (UPC), Barcelona, Spain}
\footnote{Centre de Recerca Matemàtica,  Barcelona, Spain
	Email: 
\texttt{agarmendia@crm.cat}} 
	\and 
	Sylvie Paycha\footnote{Institut für Mathematik, University of Potsdam, Potsdam, Germany
		Email: 
		\texttt{paycha@math.uni-potsdam.de}} 
}
 
\begin{document}

\date{\today}

\maketitle
\begin{abstract}
{ We consider groupoids in the category of principal bundles, which we call principal bundles (PB) groupoids. Inspired by work by Th. Nikolaus and K. Waldorf, we generalise bundle gerbes over manifolds to bundle gerbes over groupoids and discuss  a functorial correspondence between PB groupoids and bundle gerbes over groupoids. From a PB groupoid over a fibre product groupoid, {we build} a bundle gerbe over another fibre product groupoid. Conversely, from a bundle gerbe over a Lie groupoid, we build a  PB groupoid.  It has a trivial base and  from any  PB groupoid with trivial base, we build a bundle gerbe over a Lie groupoid.
	{ 	In that case,    the resulting  bundle gerbe is isomorphic   as a groupoid to a partial quotient  of the PB groupoid.}  We describe the nerves of PB groupoids and their partial quotients, which are simplicial objects in the category of principal bundles. Applying this construction enables us  to define the inner transformation group of the nerve of a partial quotient groupoid and to describe the  { transformations of the corresponding} bundle gerbe.}
\end{abstract}

 \tableofcontents
 
 \section*{Preamble}
 
 We choose to work with left group actions for principal bundles. The main reason for this choice is aesthetic since left actions combine nicely with the groupoid composition, taking an arrow from right to left in contrast with \cite{NW}, where they are taken from left to right.
Nevertheless, our conventions for bundle gerbes over Lie 2-groups coincides with those of \cite{NW}.

 \section*{Introduction}
 
{ This is an exploratory paper prompted by the wish to investigate higher analogs of gauge groupoids. Our interest for this class of groupoids stems from   the fact that they correspond to the Lie groupoids that  host direct connections studied in \cite{ABFP}. }

We consider principal bundles groupoids, namely groupoids in the category of principal bundles. Principal bundles groupoids are of the form $P^{(1)}\fto M^{(1)}$, with  $P^{(1)}\soutar P$ and $M^{(1)}\soutar M$ two Lie groupoids, with a structure group given by a Lie group groupoid ${G^{(1)}}= H\rtimes G\soutar G$  (Definition \ref{defn:Lie3pb}). Their structure group, also called a Lie $2$-group, can also be viewed as a crossed module $H\rtimes G$  built from the action of a Lie group $G$ on another Lie group  $H$ and a homomorphism $d:H\fto G$. 

{Our definition differs from other known notions of principal bundles that involve the action of a groupoid,   such as principal groupoid bundles arising in \cite{MM}, principal bundles over groupoids in \cite{LTX}, or principal 2-bundles  considered in \cite{NW} ({these last ones can nevertheless be viewed as  principal bundle groupoids, see Proposition \ref{prop:kw.to.us}.}), hence the terminology {  ``principal bundle groupoids''} we adopt here. The  main differences lie in the fact that 
	\begin{itemize}
		\item  we allow  $M^{(1)}$ to be any groupoid over $M$, not only the identity groupoid,
		\item we allow $G^{(1)}$ to be any Lie 2-group, generalising other approaches that only allow Lie groups,
		\item the action we consider is not a groupoid action on a manifold but a Lie 2-group action on a groupoid.
	\end{itemize}}

 We  consider the partial quotient    by the action on $P^{(1)}$ of the identity bisection of $H\rtimes G$, i.e. by $\{e\}\rtimes G$, and use the shorthand notation ${P^{(1)}} /_G$ (Def. \ref{eq:isoPG})  justified by:
$$ {P^{(1)}} /_G:=(P^{(1)}/\{e\}\rtimes G)\soutar M. $$

  { Partial quotients $P^{(1)}/_G$ are relevant for the study of Lie 2-group bundle gerbes over a Lie groupoid $Y^{(1)}\soutar Y$. These generalise  the classical notion of Lie 2-group bundle gerbe over a smooth manifold \cite[Definition 5.1.1]{NW} .   Specialising to bundle gerbes over fibre bundle Lie groupoids   $Y^{[2]}\soutar Y$ yields back the classical Lie 2-group bundle gerbe over the orbit space  $M\simeq Y/Y^{[2]}$.  From a principal  Lie 2-group bundle groupoid  $P^{(1)}\soutar P$  over a fibre bundle groupoid  $Y^{[2]}\soutar Y$, there is a bundle gerbe  over the fibre bundle groupoid $P^{[2]}\soutar P$, which yields a  Lie 2-group  bundle gerbe over the quotient space $P/P^{[2]}\cong M$. We also build  a Lie 2-group  bundle gerbe  over $Y^{(1)}$ from a   Lie 2-group  principal bundle groupoid  $P^{(1)}\soutar P$ over the Lie groupoid $Y^{(1)}\to Y$, with trivial base,  namely such that $P\simeq Y\times G$.   In  both cases, the resulting  bundle gerbes are Morita equivalent to the partial quotient ${P^{(1)}} /_G$. When $P^{(1)}\soutar P$ has trivial base, the resulting bundle gerbe is  actually isomorphic to $P^{(1)}/_G$  as a groupoid.

Finally, we discuss the   inner transformations of a principal $H\rtimes G$-bundle groupoid   $P^{(1)}\to M^{(1)}$ as well as  inner transformations of its partial quotient  $P^{(1)}/_G$, then generalizing them to the $k$-nerves of the corresponding objects. In particular, this  applies to inner transformations of the bundle gerbe built from a  principal  Lie 2-group bundle groupoid $P^{(1)}\soutar P$ with trivial base.}
 
 Let us briefly describe the structure of this note. 
 
 \begin{itemize}
 	\item In Section \ref{sec:PBgpd}, we review some known facts on Lie $2$-groups and define principal bundle groupoids, relating them to affine concepts in the literature, such as in \cite{NW}, \cite{LTX} and \cite{MM}.   Proposition  \ref{prop:kw.to.us} describes a principal $2$-bundle on a manifold $M$ as defined  in \cite{NW} as  a principal bundle groupoid over a Lie groupoid which is Morita equivalent to $M$.
 
 	\item   Section \ref{sec:PBGtoBG} gives a functorial correspondence, {in part inspired by \cite{NW}}, between certain classes of PB groupoids and bundle gerbes over groupoids (see Definition \ref{defn:bundlegerbeovergrp}). From a bundle gerbe $B\to Y^{(1)}$ over a groupoid $Y^{(1)}\to Y$, we build a trivial base groupoid $P^{(1)}\soutar P$ (Proposition \ref{prop:btpb.to.grb}) which can in turn be sent to a bundle gerbe $B\simeq P^{(1)}/G$ {(Proposition \ref{prop:grb.to.pb})}. 
 {In Theorem \ref{thm:pb.to.grb}, starting from a PB groupoid over $Y^{[2]}=Y\times_M Y$ we build a bundle gerbe over $P\times_M P$: these two groupoids are Morita equivalent to the identity groupoid over the manifold $M$.}
 	
 	 \item In Section \ref{sec:PBnerve}, we  describe the nerve $N^\bullet({P^{(1)}})$ of a Principal bundle groupoid ${P^{(1)}}$ as a simplicial set of regular principal bundles. Where, $G^{(k)}\curvearrowright P^{(k)}\fto M^{(k)}$ (Theorem \ref{prop:princ.2.bnd}). 
 We further interpret the  nerve $N^\bullet({P^{(1)}} )$ of  ${P^{(1)}}$   as a    principal $G $-bundle  over the nerve $N^\bullet({ {P^{(1)}} /G})$ of the quotient groupoid (Proposition \ref{prop:NerveG}).  
 
  \item Section \ref{ec:innertrans} is dedicated to inner automorphisms of PB groupoids and their nerves. Let us denote by ${\rm Mor}({P^{(1)}})$, the morphisms of principal bundle groupoid over the identity on its base ${{M^{(1)}}}$. We get an isomorphism for  nerves which makes clear that any morphism is an automorphism (see (\ref{eq:AutNP})): 
   $$C^\infty_{H^\bullet \rtimes G}  ({P^{\bullet}}, H^\bullet \rtimes G)\cong {\rm Mor}({P^{\bullet}})= {\rm  Aut} (P^{\bullet}).$$
 	 The previous results  lead to a natural definition (Definition \ref{defn:MorNG})  of the set ${\rm  Aut}({ {P^{\bullet}} /_G}  )$  of morphisms of the nerve  of  the partial quotient groupoid ${ {P^{(1)}} /_G}$  in terms of the group ${\rm  Aut}(P^{\bullet})$ and by Theorem \ref{thm:MorNGP2} we have the following group isomorphism:
 	\[ {\rm  Aut} ( P^{\bullet} /_G)\simeq 
 	\CI_{H^\bullet\rtimes G}( {P^{\bullet}}, H^\bullet).\]
Then, specialising to case of the (trivial base) PB  groupoid $P^{(1)}=\Psi(B)$ built from a bundle gerbe $B$ as in Proposition \ref{prop:grb.to.pb}. In this case and $B\simeq P^{(1)}/_G$ which yields a decription of the automorphim group of a bundle gerbe (Corollary \ref{cor:AutBk}):
\[{\rm  Aut}({B^\bullet})\simeq{\rm  Aut}(P^{\bullet}/_G)\simeq C^{\infty}_{H^\bullet\rtimes G} ({P^\bullet}, H^\bullet).\]

We end this paper with an example, specialising the above to  the pair groupoids $P^{(1)}:={\CP^{(1)}} (P)$, resp.  $M^{(1)}:={\CP^{(1)}} (M)$  of a principal bundle $G\curvearrowright P\to M$.   In that case $H=G$ and the partial quotient ${ {P^{(1)}} /_G} ={\CG^{(1)}}(P)$ is the gauge groupoid of $P$, leading to Corollary \ref{cor:GaugeNGP} which shows the group isomorphism: 
	 \[{\rm  Aut}({\CG^\bullet}(P))\simeq \CI_{G^\bullet\rtimes G}({\CP^\bullet}(P), G^\bullet),\]
	
and a cannonical embeddings between the various groups as	in Theorem \ref{thm:embeddingsgaugetransf}:
	\[ {\rm Aut}_{M} (P)  \overset{ \hat{}}{\hookrightarrow} {\rm Aut} ({\CP^{(1)}}(P)) \overset{ N^\bullet}{\hookrightarrow}   {\rm Aut} \left(  {\CP^\bullet}(P) \right) \simeq \CI_{G^\bullet\rtimes G} ( \CP^\bullet (P), G^\bullet\rtimes G). \]

\end{itemize}

With PB groupoids at hand, the next step would be to equip them with (infinitesimal) connections, which is the object of ongoing work by the first author.

\subsection*{Acknowledgments}
{We are very grateful to Sara Azzali for many fruitful discussions at various stages of the paper. We were also inspired by  prior conversations on higher gauge theory with Alessandra Frabetti, whom we would also like to thank.}

This work is supported by the Spanish State Research Agency, through the Severo Ochoa and María de Maeztu Program for Centers and Units of Excellence in R\&D (CEX2020-001084-M)
 
\section{Principal bundle groupoids} \label{sec:PBgpd}

{ In this section, after some necessary prerequisites, we define principal bundle groupoids.}\\

Let us recall that a groupoid is a small category $\CM$ in the category  SETS of sets such that every arrow is invertible. This means that the arrows $M^{(1)}:={\rm Arr}(\CM)$ (also called morphisms for a big category) and the points $M:={\rm Ob}(\CM)$ (also called objects for a big category) are both sets and the correspondences:

\begin{itemize}
	\item source and target $\bs:M^{(1)}\fto M$ and $\bt:M^{(1)}\fto M$,
	\item composition $\circ\colon M^{(1)} {}_\bs\!\times_\bt M^{(1)} \fto M^{(1)}$,
	\item  identity and inverse $e:M\fto M^{(1)}$ and $(-)^{-1}:M^{(1)}\fto M^{(1)}$,
\end{itemize}
are all maps (morphisms in SETS).\\

A Lie groupoid is a groupoid $\CM$ in the   category of manifolds. That means that $M^{(1)}$, $M$ are manifolds and $(\bs,\bt,\circ,e,(-)^{-1})$ are smooth maps (the condition of $\bs,\bt$ being submersion is required to make sense for $\circ$ being smooth). In the rest of this paper we will denote groupoids as $M^{(1)}\soutar M$, the double arrows stands for the source and target maps.\\

{In this section, we introduce groupoids in the category of principal bundles, which we call  {\bf  principal bundle groupoids} or equivalently, for a Lie 2-group $G^{(1)}$, a principal ${G^{(1)}}$-bundle groupoid or a principal bundle groupoid with structure group ${G^{(1)}}$. 
	
	\subsection{Review of Lie $2$-groups}

	\begin{defn}\label{defn:Lie2group}
		A {\bf  Lie 2-group} is a groupoid in the category of Lie groups. This is a small category  whose space of arrows $G^{(1)}$ and space of points $G$ are both Lie groups and whose groupoid  structure  maps   $(\bs,\bt,\circ,e, (-)^{-1})$ are Lie  group morphisms (homomorphisms).
	\end{defn}
	\begin{ex}	Any Lie group $G_0$ is a Lie 2-group where $G^{(1)}:=G_0$, $G:=G_0$, the group structure is given by $G_0$ and the groupoid structure is given by: $\bs=\bt=e=(-)^{-1}=\circ=I_G$. The tangent space $TG$ is also a Lie $2$-group (more details in Example \ref{ex:TG}).
	\end{ex}
	Since the category of Lie groups is included in the category of smooth manifolds,  any  Lie 2-group is a Lie groupoid .
	
	\begin{defn}\label{defn:cross module}
		A {\bf  (smooth) crossed module over the category of groups} is a pair of Lie groups $(H,G)$ together with  an  action (by homomorphisms) $C\colon G\ni g\mapsto C_g\in  {\rm Aut}(H)$ of $G$ on $H$ and a map $d:H\fto G$ { which defines an action (by diffeomorphisms) of $H$ on $G$ by $(h,g)\mapsto d(h)\cdot g$, with compatibility conditions between the two maps $C$ and $d$:  }
		\begin{eqnarray}\label{eq:crossedmodule}
			\bullet\,\, {d}(C_g h)= g \, d(h)\, g^{-1} \,\, ,&\hspace{1in} & \bullet \,\, C_{d(h)}h'=h h' h^{-1}\,, \nonumber
		\end{eqnarray}
		for all $g$ in $G$ and $h,h'$ in $H$.
	\end{defn} 
	\begin{rk}
		As we shall see, $d$ defines a groupoid structure, $C$ a group structure.
	\end{rk}
	There is a  categorical equivalence between (smooth) crossed modules over groups--a notion initially introduced by Whitehead in \cite{W1, W2}-- and (Lie) $2$-groups, also  called  (Lie)  group groupoids, see \cite{BS}, { \cite{BH}, \cite{BL}}. Let us recall this equivalence in terms of semi-direct products of groups.
	
	\begin{prop} \label{prop:cr.to.grpd}(Lie $2$-group of a crossed module) Let $(H,G,C,d)$ be a (smooth) crossed module over the category of groups.
		\begin{enumerate}
			\item   $H\times G$ is the set of arrows of a Lie groupoid over $G$ with   source and target given by $\bt(h,g):=d(h)g$, $\bs(h,g):=g$ and the groupoid composition defined as:
			$$\textnormal{\bf Groupoid Str: }  \left( h_2,{\bt( h_1, g_1)}\right)\circ (h_1,g_1):=(h_2 h_1, g_1).$$
			\item The groupoid $H\times G\soutar G$, as in item 1. above, is a Lie $2$-group with the semi-direct product:
			$$\textnormal{\bf Group Str: } (h_2,g_2)\cdot (h_1,g_1):=(h_2\,C_{g_2}h_1  , g_2 g_1),$$
			as the product on the arrows and the natural product in $G$ on the points.
		\end{enumerate}
	\end{prop}
	\begin{prop}\label{prop:grpd.to.cr} (Crossed module of a Lie $2$-group) Let $G^{(1)}\soutar G$ be a Lie $2$-group and $H:=\ker(\bs)$. 
		\begin{itemize}
			\item The group $G$ acts on $H$ via homomorphisms { $C_gh:={e(g)\, h\, e(g)^{-1}}\in H$ for all $g\in G $};
			\item $(H,G, C,d:= \bt|_H)$ is a (smooth) crossed module in the category of groups.
		\end{itemize}
	\end{prop}
	\begin{proof}
		Note that $\bs\left(C_gh\right)=\bs\left(e(g)\right)\cdot \bs(h)\cdot \bs\left(  e(g)\right)^{-1}=\bs\left(e(g)\right) \cdot \bs\left(  e(g)\right)^{-1} = 1$ which implies that $C_g(h)$ lies in $H={\rm Ker}(\bs)$.
		
		Also, clearly, we have $t(C_gh)= e(g) t(h) e(g)^{-1}$ and 
			\begin{eqnarray*}
				\left(C_{t(h)} h', t(h)\right)=(e, t(h))\cdot (h',e)&=&\left((h,e)\circ (h^{-1}, t(h))\right)\cdot \left((h', e)\circ (e,e)\right)\\
				&=& (hh', e)\circ (h^{-1}, t(h))=(h h'h^{-1}, t(h))
			\end{eqnarray*} 
		so that Conditions (\ref{eq:crossedmodule}) are verified. \end{proof}
	
	\begin{prop}\label{sec:Lie-2g}{  (see e.g.  \cite{BL}  and references therein)}
		Let $G^{(1)}\soutar G$ be a  Lie 2-group, $(H,G, C, t|_H)$ the crossed module of Prop. \ref{prop:grpd.to.cr} and $H\times G \soutar G$ the Lie $2$-group of Prop. \ref{prop:cr.to.grpd}. Then the map:
		$$\varphi\colon H\rtimes G \fto G^{(1)};(h,g)\mapsto h \, e(g)$$
		is a Lie $2$-group isomorphism (a group and Lie groupoid isomorphism).
	\end{prop}
	
	\begin{rk}
		The isomorphism $\varphi$ is given by a known construction. If $\bs:G^{(1)}\fto G$ is a homomorphism with splitting $e:G\fto G^{(1)}$ there is an isomorphism $\varphi\colon \ker(s)\rtimes G\fto G^{(1)}$.
	\end{rk}
	
	Using the above identifications, we  write any  Lie 2-group as $H\rtimes_C G\soutar G$ where $G,H$ are Lie groups.  In the sequel, for the sake of simplicity, we shall often drop the subscript $C$.

\begin{ex}
	Let $H$ be any normal subgroup of a Lie group $G$. The conjugate action $C\colon g\mapsto \left(C_g\colon h\mapsto g hg^{-1}\right)$   of $G$  on $H$ and the inclusion map $d\colon H\fto G$  define a crossed module and therefore a Lie 2-group $H\rtimes G\soutar G$. When $H=G$ one gets a Lie 2-group $G\rtimes G\soutar G$, which is isomorphic to the pair groupoid of $G$.
\end{ex}
\begin{ex}\label{ex:TG}
	Let $G$ be a Lie group and $V$ a vector space ($V$ is also a Lie group with the sum) carrying a	representation $C\colon  G\to {\rm Aut}(V)$ of $G$. The action  combined with the map $d\colon  V\to G$, which sends $v$ to $e_G$ defines a crossed module and hence a Lie 2-group  $V\rtimes_C G\soutar G$. 
	
	The above applied to $V={\mathfrak g}$ and $C$ the adjoint action  Ad of $G$ on $\mathfrak g$ equips $ \g\rtimes_{\rm Ad}G $ with a Lie $2$-group structure. 	Since the tangent bundle $TG$ of any Lie group $G$ is isomorphic to $\g\rtimes_{\rm Ad} G$, $TG$ inherits  a Lie $2$-group structure.
\end{ex}
\begin{ex}
	 The action of ${\rm O}(2k)$ on the Clifford algebra ${\rm Cliff}(2k)$ gives an action of ${\rm O}(2k)$ in ${\rm Pin}(2k)$ by homeomorphisms. The mentioned action and the covering projection $d\colon {\rm Pin}(2k)\fto {\rm O}(2k)$ defines a crossed module and therefore a Lie $2$-group ${\rm Pin}(2k)\rtimes {\rm O}(2k) \soutar {\rm O}(2k)$.
\end{ex}

\subsection{From Lie $2$-group actions to principal bundle groupoids}

{ In this paragraph we define principal Lie $2$-group bundle groupoids which involve  group actions.}

Let us now recall the action  of a Lie $2$-group  on another Lie groupoid.

\begin{defn}\label{defn:Lie2geoupaction}
	A \textbf{Lie 2-group action} of a Lie 2-group $(H\rtimes G){\soutar}G$ on a Lie groupoid $P^{(1)}\soutar P$ consists of   Lie {group} actions of $H\rtimes G$ on the manifold $P^{(1)}$ and of $G$ on the manifold $P$ such that the action map
	
	{ 	\[\begin{tikzcd}
			(H\rtimes G) \times P^{(1)}  \ar[d, shift right=.2em, swap,"t\times t"]\ar[d,shift left=.2em,"s\times s"] \ar[r] & P^{(1)} \ar[d, shift right=.2em, swap,"t"]\ar[d,shift left=.2em,"s"] \\
			G\times P \ar[r] & P
		\end{tikzcd}\]}
	is a Lie groupoid map from the cartesian product groupoid $(H\rtimes G)\times P^{(1)}$ to the Lie groupoid $P^{(1)}$.
\end{defn}
\begin{ex}\label{eq:H=e}
When $H=	\{e_H\}$, the group action of the Lie $2$-group   $H\rtimes G\simeq G$ on $P^{(1)}$ boils down to the mere  group action of $G$.
\end{ex}

Just as not every Lie  group $G$ acting on a manifold $P$ yields a quotient manifold $P/G$,    the action of a Lie $2$-group $H\rtimes G$ on a Lie groupoid ${P^{(1)}}$ does not {  always} give rise to a Lie groupoid structure on the quotient $P^{(1)}/(H\rtimes G)$. Yet, {   as in the case of Lie group actions on manifolds,} it does under a freeness and properness assumption on the action.

\begin{prop}\label{prop:lie2grp.q}
	{\cite[\S 5.2 and references therein]{G}} 
	{Consider  a free and proper right action of a  Lie 2-group $ (H\rtimes G)\soutar G$  } on a Lie groupoid $P^{(1)}\soutar P$. 
	Define the equivalence relations
	$$R:=\{(p,gp)\in P\times P \st p\in P \y g\in G\} \textnormal{ on } P,$$  
	$$R^{(1)}:=\{(\xi,(h,g)\,\xi)\in P^{(1)}\times P^{(1)} \st \xi\in P^{(1)} \y (h,g)\in H\rtimes G\} \textnormal{ on } P^{(1)}.$$

	Then $M^{(1)}:= P^{(1)}/R^{(1)}$ and $M:= P/R$ are manifolds, and $M^{(1)}\soutar M$ acquires a canonical  Lie groupoid structure making the quotient map a Lie groupoid fibration.
\end{prop}

Let us recall that a Lie groupoid fibration is a Lie groupoid morphism $\pi^{(1)}:P^{(1)}\fto M^{(1)}$ such that the map $\pi_s\colon P^{(1)}\fto M^{(1)} {}_\bs \!\times_\pi P;\phi\mapsto(\pi^{(1)}(\phi),\bs(\phi))$ is a surjective submersion. This means that given an arrow $\gamma\in M^{(1)}$ and a point $p\in P$, over the source of $\gamma$, there is a way to lift smoothly $\gamma$ to an arrow in $P^{(1)}$ starting in $p$.  The interested reader can find further details in \cite{MK}, { see \cite{CZ} for a specific  instance of the above proposition.

{ On the grounds of Proposition \ref{prop:lie2grp.q}, we   set the following definition, which compares with \cite[\S 2.8. Categorical principal bundles]{CLS} in the category of Lie groupoids.
	\begin{defn} \label{defn:Lie3pb}
		A {\bf principal $(H\rtimes G)$ -bundle groupoid} ($P^{(1)}\soutar P$) over a Lie groupoid $M^{(1)}\soutar M$ {consists of} a free and proper {{\it group}} action of the  Lie 2-group $(H\rtimes G)\soutar G$ on a Lie groupoid $P^{(1)}\soutar P$ whose quotient is isomorphic to $M^{(1)}\soutar M$, corresponding to the following diagramme: 
		
		\[\begin{tikzcd}
			(H\rtimes G)\times P^{(1)} \ar[d, shift right=.2em, swap,"t\times t"]\ar[d,shift left=.2em,"s\times s"] \ar[r] & P^{(1)} \ar[d, shift right=.2em, swap,"t"]\ar[d,shift left=.2em,"s"] \ar[r] & M^{(1)} \ar[d, shift right=.2em, swap,"t"]\ar[d,shift left=.2em,"s"]\\
			G\times	P \ar[r] & P \ar[r] & M
		\end{tikzcd}\]

		We shall write for short $(H\rtimes G) \curvearrowright {P^{(1)}}\to {M^{(1)}}$.
	\end{defn}
	\begin{rk}\label{prop:PB groupoidcat}  Principal bundles groupoids are groupoids in the category of principal bundles. 
	\end{rk}
	
	For examples related to the holonomy groupoid of a singular foliation we refer to \cite{GZ19}.

\begin{ex}\label{ex:PBgauge0}
	Let $P\to M$ be a principal $G$-bundle. On the pair groupoid ${\CP^{(1)}}(P)$, the diagonal action  $ (p, q)\mapsto (g\, p, g\, q)$  and the action $(p, q)\mapsto (h\,p, q)$ for $(h, g)\in G^2$,  can be combined in terms of a 2-group action on ${\CP^{(1)}}(P)$ letting
	$G\rtimes_C G\soutar G$ act  on ${\CP^{(1)}}(P)=P^2$ by $(h,g)\cdot (p, q)= (hg\, p, g\,q) $  for $(h,g)\in G^2$.
	The principal $G$- bundle $P\to M$ therefore induces a  principal bundle groupoid   $$\left(G\rtimes G\right) \curvearrowright {\CP^{(1)}}(P)\to  {\CP^{(1)}}(M).$$
\end{ex}
\begin{ex}\label{ex:vect bundles}
	Let $\pi\colon P\to M$ be a principal $G$-bundle. The $G$ action $A\colon G\times P\fto P$ induces an infinitesimal action we denote by $a\colon\g\times P\fto TP$. For any $g\in G$ there is also the diffeomorphism $A_g\colon P\fto P;p\mapsto A(g,p):=g\,p$. Any vector bundle (in particular $TP$) is a groupoid with source and target equal to the projection ($TP\fto P$) and the composition equal to the addition. The following formula gives a Lie 2-group action of $\g\rtimes_{\rm Ad} G\soutar G$ on $TP$ (as a Lie groupoid):
	$$(w,g)\cdot v= a(w,g\, p)+ dL_g(v) \,\,\,\, {\rm for} \,\, v\in T_p P\, , \,\, w\in \g \,\, {\rm and} \,\, g\in G.$$
	The principal $G$- bundle $P\to M$ therefore induces a  principal bundle groupoid   $$\left(\g\rtimes_{\rm Ad} G\right)\curvearrowright TP\to  TM.$$
\end{ex}

 \begin{defn}  
 		A (strict) {\bf  morphism} of $H\rtimes G$-principal bundle groupoids 
\[ f\colon \left( P^{(1)}_1\to M^{(1)}_1\right)\longrightarrow \left( P^{(1)}_2\to M^{(1)}_2\right)\] 

consists of two groupoid morphisms $f_P\colon P^{(1)}_1\to P^{(1)}_2$ and $f_M\colon {M^{(1)}_1}\to M^{(1)}_2$ such that $f_P$ is $(H\rtimes G)$-equivariant and $f_M$ is $(H\rtimes G)$-invariant.

A principal bundle groupoid $G^{(1)} \curvearrowright {P^{(1)}}\to {M^{(1)}}$ is {\bf trivial} if { there is an isomorphism of groupoids ${P^{(1)}}\simeq {G^{(1)}}\times {M^{(1)}}$ where the r.h.s is the cartesian product groupoid.  }
\end{defn}

{ \begin{rk} Strict morphisms defined as above extend to ones from a principal $(H_1\rtimes G_1)$- bundle groupoid to a principal $(H_2\rtimes G_2)$- bundle groupoid  above a Lie $2$-group morphism  $H_1\rtimes G_1\longrightarrow H_2\rtimes G_2$,  but we shall only refer to the case when $H_i\rtimes G_i=H\rtimes G, i=1,2$ with the identity Lie $2$-group morphism, hence the above definition.\end{rk}}
  
\subsection{{The partial quotient groupoid of a PB groupoid}}

We now consider   the following groupoid:

\begin{defn} \label{def:gauge.grpd} Let  $(H\rtimes G)\soutar G$ be a Lie 2-group and $P^{(1)}\soutar P$ be a $(H\rtimes G)$-principal bundle groupoid over $M^{(1)}\soutar M$. We define the {\bf  partial quotient groupoid} of $P^{(1)}$ as the quotient  by the identity bisection action of $H\rtimes G$ i.e. $P^{(1)}/(\{e_H\}\rtimes G)$ which we will denote simply by:
	\begin{equation}\label{eq:isoPG} P^{(1)}/_G \soutar \left( P/_G \cong M \right)
	\end{equation}
\end{defn}

\begin{prop}\label{prop:gaugeprincipal}{  The canonical projection map
		$Q\colon  {P^{(1)}}\longrightarrow (P^{(1)}/_G)$ is a  principal $G$-bundle and a  principal $G$-bundle groupoid.}\end{prop}
\begin{proof} 
		{  This follows from Proposition \ref{prop:lie2grp.q} with $H=\{e_H\}$.} {The fact that it is a principal $G$-bundle follows from the fact that for $H=\{e_H\}$, the  group action of    $H\rtimes G$  amounts to the group  action of   $G$ (see Example \ref{eq:H=e}).}
\end{proof}

\begin{ex}\label{ex:PBgauge}
	 Given a principal $G$-bundle $\pi_0\colon P\to M$, let  ${P^{(1)}}=P\times P$ be the pair groupoid of $P$. The {\it partial quotient groupoid of the principal bundle groupoid} in Example \ref{ex:PBgauge0} is called the gauge groupoid $\CG(P)\soutar M$ of $P\to M$ (\cite[Example 1.1.15]{MK} and references therein).
\end{ex}

\subsection{Groupoid principal  bundles arising from groupoid actions}
 {   In the previous paragraphs, we considered the group action of a Lie $2$-group. Instead, one can consider its groupoid action leading to     principal bundles stemming from  the action of a Lie groupoid on a manifold  \cite[\S 2.2]{H},    \cite[\S 5.7]{MM}, see also \cite[Definition 2.2.1]{NW}, and   Definition \ref{defn:Groupoidprincipalbundle} below.  This approach is useful since principal  bundle groupoids  relate to bundle gerbes, see  Section \ref{sec:PBGtoBG} below.}

\begin{defn}\label{def:groupoidaction}
	A left {(resp. right)} {\bf action  of a Lie groupoid} ${G^{(1)}}\overset{\bt}{\underset{\bs}{\soutar}} {G}$  on a manifold $B$ is given by a surjective submersion $\rho\colon B\to {G}$ called anchor map, and a smooth map \[\star\colon {G^{(1)}} {}_\bs \!\times_\rho B \to B\quad ({ {\rm resp.}  \star\colon B {}_\rho \!\times_\bt {G^{(1)}} \to B)}\] called the action, as in the following diagram {(resp.  for the left action)}:    
	
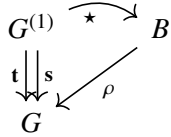
\[\begin{tikzcd}
	G^{(1)} \ar[d, shift left=.2em,"\bs"]\ar[d, shift right=.2em, swap,"\bt"] \arrow[r, bend left=30,swap, "\star"] & B \ar[dl, "\rho"]   \\
	{G} & 
	\end{tikzcd} \] 
	such that, for any $b$ in $B$, any  composable $ \gamma_2, \gamma_1 $ in ${G^{(1)}}$, we have
	\begin{enumerate}
		\item $\rho(\gamma_1\star b)= \bt(\gamma_1)$  { ( resp. $\rho(b\star \gamma_1)= \bs(\gamma_1)$)}
		\item $\gamma_2\star(\gamma_1\star b))=(\gamma_2\circ \gamma_1)\star b$ { (resp. $ (b\star \gamma_1))\star \gamma_2=b\star (\gamma_1\circ \gamma_2)$}.
	\end{enumerate}
	{ $B$ is also called a left (resp. right) $G^{(1)}$-space.}
\end{defn} 
In an effort to clarify notation, this article uses `$\star$' to denote groupoid actions whereas `$\cdot$' denotes group actions. 

Let us recall   the notion of  principal groupoid bundle of \cite{NW}:
\begin{defn} \label{defn:Groupoidprincipalbundle}\cite[\S 2.2]{H}, see also \cite[Definition 2.2.1]{NW} Let $G^{(1)}\soutar G$ be a Lie groupoid and $M$ be a smooth manifold. A {\bf principal  $G^{(1)}$-{  bundle}} (left (resp.  right) bundle) on  ${ M}$ consists of  {a groupoid action on a smooth manifold $B$ 
\[\star\colon {G^{(1)}} _\bs \!\times_\rho B \to B\quad ({ {\rm resp.}  \star\colon B {}_\rho \!\times_\bt {G^{(1)}} \to B)}\] 
 with anchor map $\rho:B\fto G$, together with a surjective submersion $\Pi:B\fto {M}$, as illustrated by the following diagram}:
	
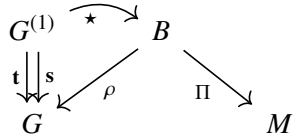
\[\begin{tikzcd}
	G^{(1)} \ar[d, shift left=.2em,"\bs"]\ar[d, shift right=.2em, swap,"\bt"] \arrow[r, bend left=30,swap, "\star"] & B \ar[dl, "\rho"] \ar[dr,swap, "\Pi"]&  \\
	{G} & & { M}
	\end{tikzcd} \] 
	such that:
	\begin{itemize}
		\item The quotient of $B$ by the action of $\star$ is a manifold $B/G^{(1)}$. 
		\item  $\Pi$ is $\star$-invariant  and  its induced map from $B/G^{(1)}$ to $M$ is a diffeomorphism.
	\end{itemize}
\end{defn} 

\begin{rk}  In \cite[\S 5.7]{MM}, \cite{H},  see also \cite[Definition 2.1]{R}, \cite[Definition 3,1]{I},  the authors call {\bf  principal groupoid bundle}  a principal $(G^{(1)}\soutar G)$-bundle. We prefer the {latter terminology} to avoid confusion with principal bundle groupoids  {as   in Definition  \ref{defn:Lie3pb}}.
	
\end{rk} 

\begin{rk}
	The  equation  $g\cdot p = p \cdot (g^{-1})$ transforms a right action into a left action and vice-versa. So, without loss of generality, we can use principal $G^{(1)}$-bundles induced by left or right actions.
\end{rk}

 Principal bundle groupoids {as defined in Definition \ref{defn:Lie3pb}, which involve the group action of a Lie $2$-group,  can be viewed as a smooth version of  principal categorical bundles discussed in  \cite{CLS}.
 
Yet they} differ from other notions of principal bundles involving the action of a groupoid. The  main differences lies in the fact that:
	\begin{itemize}
		\item  we allow  $M^{(1)}$ to be any groupoid over $M$, not only the groupoid of identities.  
		 In \cite[Definition 6.1.5.]{NW}, the authors define  {\bf  principal 2-bundles}   with  a manifold as the base in constrast with our   principal bundle groupoids whose base {space is any groupoid as in Definition  \ref{def:gauge.grpd}}.
		
		\item We allow $G^{(1)}$ to be any Lie 2-group, generalising other approaches  that only allow Lie groups. In \cite[Definition 2.2]{LTX},   the authors define {\bf principal bundles over groupoids} from a free and proper Lie 2-group action of the identity Lie 2-group $G\soutar G$ in a groupoid $P^{(1)}$ with quotient a Lie groupoid $M^{(1)}$. Our definition is a generalisation  which allows  $G$ to be any Lie 2-group. Note that we get a principal bundle over a groupoid when considering the partial quotient groupoid of a principal bundle groupoid as in Proposition \ref{prop:gaugeprincipal}. 
		
	\end{itemize}
 The two definitions of  principal bundle over a groupoid in \cite{LTX} on the one hand and  in \cite{MM} on the other, relate by generalised morphisms of Lie groupoids.  
 \begin{defn}\label{defn:generalisedhom}{\cite[Definition 2.8]{LTX} A {\bf  generalised groupoid homomorphism}  from $G^{(1)}\to G$ to $H^{(1)}\to H$  is given by a manifold $P$, two smooth maps $G\overset{\pi}{\leftarrow} P\overset{\rho}\rightarrow H$, a left action of $G^{(1)}$  w.r. to $\rho$, a right action of $H^{(1)}$ with respect to $\pi$, such that the two actions
 		commute, and $P/H^{(1)}\simeq G$.
 	  	\[\begin{tikzcd}
 		G^{(1)} \ar[d, shift left=.2em,"\bs"]\ar[d, shift right=.2em, swap,"\bt"] \arrow[r, bend left=30,swap, "\star"] & P \ar[dl, "\pi"] \ar[dr,swap, "\rho"]& H^{(1)} \arrow[l, bend right=30,  "\star"] \ar[d, shift left=.2em,"\bs"]\ar[d, shift right=.2em, swap,"\bt"]  \\
 		{G} & & { H}\\
 		\end{tikzcd} \] } 
 \end{defn}

 Let $H$ be a Lie group. A principal $H$-bundle over a groupoid    $G^{(1)}\soutar G$ as in \cite[Definition 2.2]{LTX}  (the authors consider right actions) corresponds to a generalised morphism from $G^{(1)}$ to $H\soutar \{*\}$ (there is an equivalence of categories \cite[Proposition 2.11]{LTX}).

\subsection{{$G^{(1)}$-principal $2$-bundles in terms of principal bundle groupoids}}

We now  consider a right group action of a Lie $2$-group instead of the above groupoid action.  We shall compare Definition \ref{defn:Lie3pb} with the subsequent definition.

  We first recall the notion of weak equivalence. 
\begin{defn}\label{defn:Mor3}\cite[Theorem 2.3.13]{NW} Let $P^{(1)}\soutar P$ and $M^{(1)}\soutar M$ be two Lie groupoids. A smooth functor  $\varphi^{(1)}\colon P^{(1)}\longrightarrow M^{(1)}$   is a  {\bf weak equivalence} if:
	\begin{itemize}
		\item $\varphi^{(1)}$ is full and faithful, in other words,  for any $p,q\in P$, the map $\pi\colon P^{(1)}(p,q)\fto M^{(1)}(\varphi(p),\varphi(q))$ is bijective.
		\item $\varphi^{(1)}$ is essentially surjective, in other words,   the map $M^{(1)} {}_\bs \! \times _\varphi P\fto M; (\alpha,p)\mapsto \bt(\alpha)$ is a surjective submersion.
	\end{itemize}  
\end{defn}

 \begin{lem}\label{lem:weakequiv}
		Given a fibration $\pi^{(1)}:P^{(1)}\fto Y^{(1)}$ and a surjective morphism $\varphi^{(1)}:Y^{(1)}\fto M$, the inclusion $P^{(1)} \times_{Y^{(1)}} P^{(1)} \fto P^{(1)} \times_{M} P^{(1)} $ is a weak equivalence iff $\varphi^{(1)}$ is a weak equivalence.
	\end{lem}\begin{proof}  
{We need to prove that } the canonical injection	$\iota\colon P^{(1)} \times_{Y^{(1)}} P^{(1)} \fto P^{(1)} \times_{M} P^{(1)} $ is full, faithful and essentially surjective, if and only if $\pi:Y^{(1)}\fto M$ is also full, faithful and essentially surjective.

\begin{itemize}
	\item  If the canonical inclusion	$\iota\colon P^{(1)} \times_{Y^{(1)}} P^{(1)} \fto P^{(1)} \times_{M} P^{(1)} $ is full, faithful and essentially surjective, then $\varphi^{(1)}:Y^{(1)}\fto M$ is also ful, faithful and essentially surjective. This can be read off  the following commutative diagram, which shows that the  { bottom projection is not full,  faithful and essentially surjective then the top inclusion} cannot be:	
$$	\begin{tikzcd}
	P^{(1)} \times_{Y^{(1)}} P^{(1)}\ar[r, "\iota"]  \ar[d]& P^{(1)} \times_{M} P^{(1)} \ar[d] \\
	Y^{(1)} \ar[r, "\varphi^{(1)}"]  & M
	\end{tikzcd}$$

	\item
	
Assuming that $\varphi^{(1)}:Y^{(1)}\fto M$ is full, faithful and essentially surjective, we now prove that $\iota\colon P^{(1)} \times_{Y^{(1)}} P^{(1)} \fto P^{(1)} \times_{M} P^{(1)} $ is full and faithful. For any $(\alpha,\beta)\in P^{(1)} \times_{M} P^{(1)}$ there is only one $\gamma\in Y^{(1)}$ with $\varphi^{(1)}(\gamma)=\varphi^{(1)}(\pi^{(1)}(\alpha))=\varphi^{(1)}(\pi^{(1)}(\beta))$ which means that $\pi^{(1)}(\alpha)=\pi^{(1)}(\beta)=\gamma$ and $(\alpha,\beta)\in P^{(1)} \times_{Y^{(1)}} P^{(1)}$.\\
Let us check  that the map $\iota\colon P^{(1)} \times_{Y^{(1)}} P^{(1)} \fto P^{(1)} \times_{M} P^{(1)} $ is also essentially surjective. The map $\bt:P^{(1)} \times_{M} P^{(1)}  \fto P \times_{M} P$ is a surjective submersion by definition. To prove the statement, we now need to show that around any $(\alpha,\beta)\in P^{(1)} \times_{M} P^{(1)}$, there is a smooth choice $\gamma_{\alpha\beta}\in P^{(1)}$ such that $(\alpha\circ \gamma,\beta)\in P^{(1)} \times_{M} P^{(1)}$ and $\bs(\alpha\circ \gamma,\beta)\in P\times_Y P$. Since $\varphi(\bs(\pi^{(1)}(\alpha)))=\varphi(\bs(\pi^{(1)}\beta))$ and $\varphi^{(1)}$ is full and faithful,  there is only one smooth choice $\omega_{\alpha\beta}\in Y^{(1)}$ from $q:=\bs(\pi^{(1)}(\beta))$ to $p:=\bs(\pi^{(1)}(\alpha))$. The desired $\gamma_{\alpha\beta}$ is any arrow in $P^{(1)}$ with target $\bt(\alpha)$ such that $\pi^{(1)}(\gamma_{\alpha\beta})=\omega_{\alpha\beta}$, this smooth choice exists since $\pi^{(1)}$ is a fibration.
\end{itemize}\end{proof}
Let us recall   the notion of  principal groupoid 2-bundle of \cite{NW}:
\begin{defn}\cite[Definition 6.1.5.]{NW}\label{defn:principal2groupbundle} Let $M$ be a smooth manifold and $({G\ltimes H})$ a  Lie $2$-group.
	
	A {\bf principal $({G\ltimes H})$- $2$-bundle over $M$} is defined from a right Lie $2$-group action \begin{eqnarray*}\star\colon {P^{(1)}}\times ({G\ltimes H}) &\fto &{P^{(1)}}\\
		(\phi,(g,h)) &\mapsto& \phi\star (g,h),
	\end{eqnarray*} on another Lie groupoid ${P^{(1)}} \soutar P$, with a groupoid morphism $P^{(1)}\fto M$, where  $M\soutar M$  is the identity groupoid over a manifold $M$, for which the morphism:
	\[\begin{tikzcd}
	P^{(1)} \times (G\ltimes H) \ar[d, shift right=.2em, swap,"t_P\times t_G"]\ar[d,shift left=.2em,"s_P\times s_G"] \ar[r,"{\rm pr}_1\times a"] & P^{(1)}\times_M P^{(1)} \ar[d, shift right=.2em, swap,"t_P\times t_P"]\ar[d,shift left=.2em,"s_P\times s_P"] \\
	P\times G \ar[r, "{\rm pr}_1\times a"] & P\times_M P
	\end{tikzcd}\]
	is a  weak equivalence.
\end{defn}} Principal $2$-bundles are, in a non trivial way,  principal bundle groupoids, as it is shown in  the subsequent proposition.

\begin{prop}\label{prop:kw.to.us} Let $G^{(1)}\soutar G$ be a { Lie} $2$-group.
	A  principal {$  G^{(1)}$}- $2$-bundle   on a manifold $M$ { as in Definition  \ref{defn:principal2groupbundle}},  is a principal $G^{(1)}$-bundle groupoid (in the sense of Definition \ref{defn:Lie3pb}) over a Lie groupoid $Y^{(1)}\soutar Y$, which is Morita equivalent to the manifold $M$. 
\end{prop}
\begin{proof} With the notations of Definition \ref{defn:principal2groupbundle}, let $P^{(1)}\fto M$ be a $G^{(1)}$-principal  $2$-bundle. We want to define a  principal  $G^{(1)}$-bundle groupoid ($P^{(1)}\soutar P$) over a Lie groupoid $Y^{(1)}\soutar Y$, which is Morita equivalent to $M$ by Example \ref{ex:Y[1]}.
	
We set $Y^{(1)}=P^{(1)}/G^{(1)}$  in which case $G^{(1)}\times P^{(1)}\cong P^{(1)} \times_{Y^{(1)}} P^{(1)}$ and the quotient map $\pi^{(1)}\colon P^{(1)}\fto Y^{(1)}$ is a fibration. Because the morphism $P^{(1)}\fto M$ is $G^{(1)}$-invariant, there exists a morphism $\varphi\colon Y^{(1)}\fto M$. By Lemma \ref{lem:weakequiv}, the inclusion $P^{(1)} \times_{Y^{(1)}} P^{(1)} \fto P^{(1)} \times_{M} P^{(1)} $ is a weak equivalence iff   $\varphi^{(1)}:Y^{(1)}\fto M$ is a weak equivalence. By Lemma \ref{lem:equivMor}, $Y^{(1)}$ is Morita equivalent to $M$  (see Example  \ref{ex:Y[1]}).\\

\end{proof}

\section{{{ A functorial correspondence between PB groupoids and bundle gerbes}}}\label{sec:PBGtoBG}
{In this section, we generalise the classical notion of Lie 2-group bundle gerbe over a smooth manifold \cite[Definition 5.1.1]{NW} to that  of a Lie 2-group bundle gerbe over a Lie groupoid $Y^{(1)}\soutar Y$.  Specialising to bundle gerbes over fibre bundle Lie groupoids   $Y^{[2]}\soutar Y$ yields back the classical Lie 2-group bundle gerbe over the orbit space  $M\simeq Y/Y^{[2]}$.  From a principal  Lie 2-group bundle groupoid  $P^{(1)}\soutar P$  over a fibre bundle groupoid  $Y^{[2]}\soutar Y$, we build a bundle gerbe  over the fibre bundle groupoid $P^{[2]}\soutar P$, which yields a bundle gerbe over the quotient space $P^{[2]}/P$.  From a trivial base Lie 2-group  principal bundle groupoid  $P^{(1)}\soutar P$ over a Lie groupoid $Y^{(1)}\to Y$, we build  a Lie 2-group  bundle gerbe  over $Y^{(1)}$.}

{Let us first recall the definition of a fibre bundle groupoid.

\begin{defn}\label{defn:fibreproductgroupoid} \cite{MK} To a { surjective} submersion $\pi:Y\to M$  between two smooth manifolds $Y$ and $M$, we associate  the \textbf{fiber product groupoid} of $\pi$,  whose set of points is $Y$ and whose set of arrows is given by: 
	$$Y^{[2]}:= Y \times_M Y=\{(y_2,y_1)\in Y^2 \st \pi(y_2)=\pi(y_1)\}.$$
	for any $(y_3,y_2),(y_2,y_1)\in Y^{[2]]}$ there is
	\[\bs (y_2,y_1):=y_1 \, \,\, , \, \,\, \bt(y_2,y_1):=y_2\, \,\, , \, \,\, (y_3,y_2)\circ (y_2,y_1) := (y_3,y_1)\,,\]
	\[e_{y_1}:=(y_1,y_1)\in Y^{[2]} \y (y_2,y_1)^{-1}:=(y_1,y_2).\]
\end{defn}}
\subsection{{ Bundle gerbes over Lie groupoids}}

\begin{defn}\label{def:bg.kw} \cite[Definition 5.1.1]{NW}  Given a Lie $2$-group ${H\rtimes G}$, 
	a {\bf  $({H\rtimes G})$-bundle gerbe} over a smooth manifold  $M$  consists of
	
	\begin{itemize}
	\item A Lie groupoid $B\soutar Y$ with composition $\circ_\mu\colon B _{\bs} \!\times_{\bt}  B\to B$;
	\item A surjective submersion $\pi\colon Y\to M$;
	\item A $({H\rtimes G}\soutar G)$ groupoid action $\star:{(H\rtimes G)} \,_\bs \!\times_\rho B\fto B $ with anchor $\rho:B\fto G$.
\end{itemize}
Such that:
\begin{enumerate}
	\item     The image of the map $ \bt\times\bs\colon B \fto Y\times Y$ is the submanifold $Y^{[2]}=Y\times_M Y$. In other words, the map $\pi$ gives a diffeomorphism between the orbit space $Y/B$ and $M$, where the orbit space $Y/B$ is the quotient of $Y$ by the equivalence relation: $y_1\sim y_2\Leftrightarrow \exists b\in B, \bs(b)=y_1\wedge \bt(b)=y_2$.
	 
	\item  The induced map {$\Pi:B\fto Y^{[2]}$} is a submersion and an  $({H\rtimes G}\soutar G)$-principal bundle with respect to the groupoid action, as the following diagram suggests:
		\[\begin{tikzcd}
			H\rtimes G \ar[d, shift left=.2em,"\bs"]\ar[d, shift right=.2em, swap,"\bt"] \arrow[r, bend left=30,swap, "\star"] & B \ar[dl, "\rho"] \ar[r, "\Pi"]& Y^{[2]} \ar[d, shift left=.2em,"\bs"] \ar[d, shift right=.2em, swap,"\bt"]&  \\
			{G} & & Y  \ar[r,"\pi"] & M
		\end{tikzcd} \] 
		\item The composition $\circ_\mu$ is $(H\rtimes G)$-equivariant, i.e.:
		$$((h_2,g_2)\star b_2) \circ_\mu ((h_1,g_1)\star b_1)= ((h_2,g_2)\cdot (h_1,g_1))\star (b_2 \circ_\mu b_2),$$
		for $(h_2,g_2),(h_1,g_1)\in H\rtimes G$ and composable elements $b_2,b_1\in B$.
	\end{enumerate}
\end{defn}
\begin{rk}
{The above definition corresponds to an interpretation of the bundle gerbe product $\mu$ in \cite[Definition 5.1.1]{NW} in terms of a groupoid composition $\circ_\mu$, in viewing $B$ as a groupoid over $Y$. The  existence of an identity and the invertibility in $B$ result from \cite[Lemma 5.2.5]{NW}.}
\end{rk}

{We generalise the above definition  to bundle gerbes over groupoids.}
\begin{defn}\label{defn:bundlegerbeovergrp}
	Given a  Lie $2$-group   $H\rtimes G$ and a groupoid $Y^{(1)}\soutar Y$, an   ${H\rtimes G}$-bundle gerbe over a Lie groupoid $Y^{(1)}\soutar Y$ consists of
	\begin{itemize} 
		\item A Lie groupoid $B\soutar Y$, with composition $\circ_\mu\colon B \,_\bs\!\times_\bt B\fto B$;
		\item A groupoid morphism $\Pi\colon B\fto Y^{(1)}$ over the identity of $Y$;
		\item A $({H\rtimes G}\soutar G)$-groupoid action  $\star:(H\rtimes G)_\bs \! \times_\rho B \fto B$ and anchor map $\rho:B\fto G$,
	\end{itemize}
such that:
\begin{enumerate} 
		\item $\Pi\colon B\fto Y^{(1)}$ is a $({H\rtimes G}\soutar G)$-principal bundle.
		\item The composition $\circ_\mu$ of $B$ is $({H\rtimes G})$-equivariant i.e.
		$$((h_2,g_2)\star b_2) \circ_\mu ((h_1,g_1)\star b_1)= ((h_2,g_2)\cdot (h_1,g_1))\star (b_2 \circ_\mu b_2),$$
		for all $(h_2,g_2),(h_1,g_1)\in H\rtimes G$ and $b_2,b_1\in B$,
	\end{enumerate}
		as suggested by the following diagram:
	\begin{equation}\label{eq:HGbundlegerbe}\begin{tikzcd}
	H\rtimes G \ar[d, shift left=.2em,"\bs"]\ar[d, shift right=.2em, swap,"\bt"] \arrow[r, bend left=30,swap, "\star"] & B \ar[dl, "\rho"] \ar[r, "\Pi"] & Y^{(1)} \ar[d, shift left=.2em,"\bs"] \ar[d, shift right=.2em, swap,"\bt"]  \\
	{G} & & Y  
	\end{tikzcd} \end{equation}

\end{defn} 

\begin{ex} Given a submersion $\pi\colon Y\fto M$ and taking 
	$	Y^{(1)}= Y^{[2]}:=Y\times_\pi Y$ to be the tensor product groupoid over $Y$, gives back a bundle gerbes over the orbit space $M=Y/Y^{[2]}$.
\end{ex}

\begin{defn} A (strict) {\bf  morphism} of $H\rtimes G$-bundle gerbes 
	\[ f\colon \left( B_1\to Y^{(1)}_1\right)\longrightarrow \left( B_2\to Y^{(1)}_2\right)\] 
	
	consists of two groupoid morphisms $f_B\colon B_1\fto B_2$ and $f_Y\colon {Y^{(1)}_1}\to Y^{(1)}_2$ such that $f_B$ is $(H\rtimes G)$-equivariant and $f_Y$ is $(H\rtimes G)$-invariant (as groupoid actions).
\end{defn}
\begin{rk}Later we shall see how a morphism of bundle gerbes arises from a morphism of principal bundle groupoids. 
\end{rk}
\subsection{From a PB groupoid over a fiber product groupoid to a bundle gerbe}
 We build a functor 
 $$\Phi:\text{PB groupoids over fibre product groupoids}\fto \text{Bundle gerbes over fiber product groupoids}$$
 
\begin{prop}\label{thm:pb.to.grb} Let $H\rtimes G$ be a Lie $2$-group, $\pi:Y\fto M$ be a submersion between the manifolds $Y$ and $M$. Given a principal {$(H\rtimes G)$-bundle groupoid ($P^{(1)}\soutar P$) over  the fiber product groupoid   $(Y^{[2]}\soutar Y)$} (see Definition \ref{defn:Lie3pb})
	\[\begin{tikzcd}
		(H\rtimes G)\times P^{(1)} \ar[d, shift right=.2em, swap,"\bt\times t"]\ar[d,shift left=.2em,"\bs\times s"] \ar[r] & P^{(1)} \ar[d, shift right=.2em, swap,"t"]\ar[d,shift left=.2em,"s"] \ar[r,"{Q}^{(1)}"] & Y^{[2]} \ar[d, shift right=.2em, swap,"t"]\ar[d,shift left=.2em,"s"]\\
		G\times	P \ar[r] & P \ar[r,"Q"] & Y
	\end{tikzcd}\]
	 the following data defines a bundle gerbe $\Phi(P^{(1)}):=G\times P^{(1)}$ over $P^{[2]}=P\times_M P$: 

	\begin{equation}\label{eq:diagrammfibration}\begin{tikzcd}
		H\rtimes G \ar[d, shift left=.2em,"\bs"]\ar[d, shift right=.2em, swap,"\bt"] \arrow[r, bend left=30,swap, "\star"] & G\times P^{(1)} \ar[dl, "\rho"] \ar[r, "\Pi"]& P\times_M P \ar[d, shift left=.2em,"\bs"] \ar[d, shift right=.2em, swap,"\bt"]&  & \\
		{G} & & P  \ar[r,"Q"] & Y \ar[r,"\pi"] & M
	\end{tikzcd} \end{equation}
where {for any $(g, \phi)$ in $G\times P^{(1)}$}:
	$$\Pi (g,\phi):= (g\cdot \bt(\phi),\bs(\phi)) \,\, ,\,\,\, \rho(g,\phi):= g,$$
	$$(h,g)\star (g,\phi):=\left(\,d(h)\,g \,\, , \,(C_{g^{-1}}h^{-1},e)\cdot \phi \,\right)\quad \text{using the notations of Definition \ref{defn:cross module}},$$
	$$(g_2,\phi_2)\circ_\mu (g_1,\phi_1):=\left(\, g_2 \cdot g_1 \,\, , \,\, \left( (e,g_1^{-1})\cdot \phi_2\right) \circ \phi_1 \,\, \right),$$
	for $(h,g)\in H\rtimes G$, $\phi\in P^{(1)}$ and composable $(g_2,\phi_2),(g_1,\phi_1)\in G\times P^{(1)}$ (i.e. $s(\phi_2))=g_1\cdot t(\phi_1)$ ).
\end{prop}

\begin{rk}Starting from a PB groupoid over $Y^{[2]}$ we get a bundle gerbe over $P\times_M P$: these two groupoids are Morita equivalent to the identity groupoid over $M$ (ee Example \ref{ex:Y[1]}).
\end{rk}

Proposition \ref{thm:pb.to.grb} is proven in \cite{NW}. We give here a proof that allow us to generalize it, as Proposition \ref{prop:btpb.to.grb}.

\begin{proof}
	
	\begin{itemize} 
		\item $\star$ is a groupoid action with anchor $\rho$. It is free because $H\rtimes G$ acts freely in $P^{(1)}$. { It is proper since $\Pi$ is a submersion and invariant under this action}.
		\item $\circ_\mu$ is associative and smooth because $\cdot$ and $\circ$ are associative and smooth.
		\item Let us prove that $\Pi$ is a surjective submersion,  which uses the fact  that $Y^{[2]}$ is a fibre product groupoid:\\
	 It follows from Definition \ref{defn:Lie3pb} that the map $Q^{(1)}:P^{(1)}\fto Y^{[2]}$,  where $Y^{[2]}$ is viewed as a groupoid $Y^{[2]}\soutar Y$  and $Q^{(1)}(\phi)=(Q(\bt(\phi)),Q(\bs(\phi)))$, is a fibration, which implies that the  map 
		$$Q^{(1)}_s:P^{(1)}\fto Y^{(1)}{}_\bs \!\times_Q P;\phi\mapsto (Q^{(1)}(\phi),\bs(\phi))$$
		is a surjective submersion. We want to show that $\Pi:G\times P^{(1)}\to P\times_MP$ is a surjective submersion.\\ 
		Hence, for any $(q,p)\in P\times_M P$ and $(V,U)\in T_{(p,q)}(P\times_M P)=T_p P \times_{TM} T_q P$ there is $\phi\in P^{(1)}$ and $W\in T_\phi P^{(1)}$ such that:
		$$ \bs(\phi)=p, \,   \y Q^{(1)}(\phi)=(Q(q),Q(p))\in Y\times_M Y,$$
		$$ d\bs(W)=X \y dQ^{(1)}(W)=(dQ(V),dQ(U))\in T_{(Q(q),Q(p))}(Y\times_M Y).$$
		This implies that:\\ 
		{ $\Pi$ is surjective}: $Q(q)=Q(\bt(\phi))$ so there is $g\in G$ such that $\Pi(g,\phi)=(g\bt(\phi),\bs(\phi))=(q,p)$.\\
	{ $\Pi$ is a submersion}: $dQ(Y) = dQ (d\bt (W))$ so there is $v\in T_g G$ such that $d\Pi(v,W)=(V,U)$.
		\item Let us prove that $\circ_\mu$ is $(H\rtimes G)$-equivariant:\\
		For any composable elements $(g_2,\phi_2),(g_1,\phi_1)\in G\times P^{(1)}$ and $h_2,h_1\in H$, { using the notations $C\colon G\to {\rm Aut}(H)$ and  $d\colon H\to G$ of Definition \ref{defn:cross module}, let us set}:
		\begin{itemize}
			\item[\textbullet] $g:=d(h_2)\, g_2 \, d(h_1) \, g_1$,
			\item[\textbullet] $({\bf h},{\bf g}):=\left(C_{g_1^{-1} d(h_1^{-1})g_2^{-1}} h_2^{-1},g_1^{-1}d(h_1^{-1})g_1 \right)$.
		\end{itemize}
		The elements $({\bf h},{\bf g})$ and $\left(C_{g_1^{-1}}h_1^{-1},e\right)$ are composable in $H\rtimes G$ { since by  (\ref{eq:crossedmodule}) we have $\bt\left(C_{g_1^{-1}}h_1^{-1},e\right)= d(C_{g_1^{-1}}h_1^{-1})e= g_1^{-1} d(h_1^{-1}) g_1 =\bs \left(\bf h, \bf g\right)$} and:
		$$({\bf h},{\bf g})\circ\left(C_{g_1^{-1}}h_1^{-1},e\right)={({\bf h}\,C_{g_1^{-1}}h_1^{-1}, e)= \left(C_{g_1^{-1} d(h_1^{-1})g_2^{-1}} h_2^{-1}\,C_{g_1^{-1}}h_1^{-1}, e\right)= \left(C_{g_1^{-1}}\left(  h_1^{-1} \left( C_{g_2^{-1}}  h_2^{-1}\right)h_1\right)\,C_{g_1^{-1}}h_1^{-1}, e\right)}$$
		$${  \left(C_{g_1^{-1}}\left(  h_1^{-1}   C_{g_2^{-1}}  h_2^{-1}\right)C_{g_1^{-1}}\left( h_1\right)\,C_{g_1^{-1}}h_1^{-1}, e\right)=} \left(C_{g_1^{-1}}\left( h_1^{-1} C_{g_2^{-1}} h_2^{-1}\right),e\right).
		$$ 
		
	Moreover, on the one hand we have:
		\begin{eqnarray*}& &\left((h_2,g_2)\star (g_2,\phi_2)\right)\circ_\mu \left((h_1,g_1)\star (g_1,\phi_1)\right)\\
		&=&{\left( d(h_2)g_2, (C_{g_2}^{-1} h_2^{-1}, e)\cdot \phi_2\right)\circ_\mu  \left( d(h_1)g_1, (C_{g_1}^{-1} h_1^{-1}, e)\cdot \phi_1\right)}\\
			&=& {\left( g, \left(\left(e, g_1^{-1}d(h_1^{-1}) \right)\cdot\left((C_{g_2}^{-1} h_2^{-1}, e)\cdot \phi_2\right)\right)\circ \left((C_{g_1}^{-1} h_1^{-1}, e)\cdot \phi_1\right)\right)}\\
				&=& {\left(g, \left(C_{g_1^{-1}d(h_1^{-1})} \left(C_{g_2}^{-1} h_2^{-1}\right), g_1^{-1}d(h_1^{-1})\right)\cdot \phi_2\right)\circ \left((C_{g_1}^{-1} h_1^{-1}, e)\cdot \phi_1\right)}\\
					&=& {\left(g, \left(C_{g_1^{-1}d(h_1^{-1})g_2^{-1}}   h_2^{-1}, {\bf g}\, g_1^{-1} \right)\cdot \phi_2\right)\circ \left((C_{g_1}^{-1} h_1^{-1}, e)\cdot \phi_1\right)}\\
			&=&\left( g \,\, , \,\, \left( ({\bf h},{\bf g}) \cdot \left((e,g_1^{-1})\cdot\phi_2 \right) \right)\circ \left(\left(C_{g_1^{-1}}h_1^{-1},e\right) \cdot \phi_1 \right) \right).\end{eqnarray*}
		 
		On the other hand, 
	\begin{eqnarray*}&&\left((h_2,g_2)\cdot(h_1,g_1)\right)\star \left((g_2,\phi_2)\circ_\mu(g_1,\phi_1)\right)\\
		&=&{(h_2\cdot C_{g_2}h_1  , g_2\cdot g_1)\star \left(\, g_2 \cdot g_1 \,\, , \,\, \left( (e,g_1^{-1})\cdot \phi_2\right) \circ \phi_1 \,\, \right)}\\
	&=&	{\left(d(h_2)\cdot d(C_{g_2}h_1) ( g_2\cdot g_1) , \,\,  \left(\left(C_{(g_2\cdot g_1)^{-1}}(h_2\cdot C_{g_2}h_1)^{-1},e\right)\cdot \left((e, g_1^{-1})\cdot \phi_2\right)\right) \circ \phi_1 \,\, \right)}\\
		&=&	{\left(d(h_2)\cdot d(C_{g_2}h_1) ( g_2\cdot g_1) , \,\,  \left(\left(C_{(g_2\cdot g_1)^{-1}}(h_2\cdot C_{g_2}h_1)^{-1},e\right)\cdot \left((e, g_1^{-1})\cdot \phi_2\right)\right) \circ \phi_1 \,\, \right)}\\
		&=& {\left(d(h_2)\cdot g_2\cdot d(h_1) \cdot g_1 , \,\,  \left(\left(C_{(g_2\cdot g_1)^{-1}}(h_2\cdot C_{g_2}h_1)^{-1},e\right)\cdot \left((e, g_1^{-1})\cdot \phi_2\right)\right) \circ \phi_1 \,\, \right)}\\
		&=&  \left(g \,\, , \,\, \left( ({\bf h},{\bf g})\circ\left(C_{g_1^{-1}}h_1^{-1},e\right)\right)\cdot\left(\left((e,g_1^{-1})\cdot\phi_2\right) \circ \phi_1\right)\right),\end{eqnarray*}
	The two expressions coincide	since  $\cdot$ is a Lie 2-group action:
	$$\left( ({\bf h},{\bf g}) \cdot \left((e,g_1^{-1})\cdot\phi_2 \right) \right)\circ \left(\left(C_{g_1^{-1}}h_1,e\right) \cdot \phi_1 \right)=\left( ({\bf h},{\bf g})\circ\left(C_{g_1^{-1}}h_1,e\right)\right)\cdot\left(\left((e,g_1^{-1})\cdot\phi_2\right) \circ \phi_1\right),$$
		proving that 
		$$\left((h_2,g_2)\star (g_2,\phi_2)\right)\circ_\mu \left((h_1,g_1)\star (g_1,\phi_1)\right)=\left((h_2,g_2)\cdot(h_1,g_1)\right)\star \left((g_2,\phi_2)\circ_\mu(g_1,\phi_1)\right).$$
	 
	\end{itemize}

This ends the proof.

\end{proof}

{The above proposition  gives rise to a functor $\Phi$ from  $(H\rtimes G)$-groupoids to  $(H\rtimes G)$-bundle gerbes   by  sending a morphism $f$ of $(H\rtimes G)$-groupoids to the morphism $ {\rm Id}_G\times f$ of the corresponding bundle gerbes.}

\subsection{From a bundle gerbe over a groupoid  to a base trivial PB groupoid}

\begin{defn}
	A {trivial base $(H\rtimes G)$-principal bundle groupoid }  is a principal bundle groupoid such that the $0$ level is a trivial principal bundle i.e., $P= G\times Y \to Y$, where as before, the source and target of the Lie 2-group $G^{(1)}= H\rtimes G\soutar G$ built from a map  $d:=\bt|_H:H\to G$,   are given by $\bs_{G^{(1)}}(h, g):=g$, $\bt_{G^{(1)}}(h, g):=d(h)g$ (see Proposition \ref{prop:cr.to.grpd}).
	
\end{defn}

We build a functor 
$$\Psi:\text{Bundle gerbes}\fto \text{base trivial PB groupoids}.$$


\begin{prop}\label{prop:grb.to.pb}
	Given a bundle gerbe $B$ over $Y^{(1)}$ described by the   diagram (\ref{eq:HGbundlegerbe}), 
	there is a base trivial PB groupoid given by:
	\[\begin{tikzcd}
	H\rtimes G \ar[d, shift right=.2em, swap,"\bt"]\ar[d,shift left=.2em,"\bs"] \ar[bend left=25, shift right=.8em,r] & G\times B \ar[d, shift right=.2em, swap,"\bt"]\ar[d,shift left=.2em,"\bs"] \ar[r] & Y^{(1)} \ar[d, shift right=.2em, swap,"\bt"]\ar[d,shift left=.2em,"\bs"]\\
	G \ar[bend left=25, shift right=.8em,r] & G\times Y \ar[r] & Y
	\end{tikzcd}\]
	where:
	\begin{itemize}
		\item For all $(k_1,b_1)\in G\times B$ we have: $\bs(k_1,b_1)=(k_1,\bs(\Pi (b_1)))$ and $\bt(k_1,b_1)=(k_1\rho(b_1),\bt(\Pi(b_1)))$.
		\item for $(k_2,b_2),(k_1,b_1)\in (G\times B) _\bs \!\times_\bt (G\times B)$ then:
		$$(k_2,b_2)\circ (k_1,b_1)= (k_1, (b_2\circ_\mu b_1)).$$
		\item for $(h,g)\in H\rtimes G$ and $(k_1,b_1)\in G\times B$ we have the Lie 2-Group action:
		$$(h,g)\cdot(k_1,b_1)=\left(\, gk_1 \, \,, \,\, \left(C_{\rho(b_1)k_1^{-1}g^{-1}} h^{-1} \,,\, \rho(b_1)\right)\star b_1 \, \right).$$
		\item {The identity morphism of an object  $(g, y)$ is $ (g, {\bf 1}_B(y))$, where ${\bf 1}_B$ is the unit element of $B$.}
		\item {	The inverse of
			a morphism $(g, b)$ is $(g\,\rho(b)^{-1}, i_B(b))$, where $i_B$ is the inverse map of $B$.} 
	\end{itemize}
\end{prop}
\begin{proof}{  The proof follows the construction of principal 2-bundles  from    bundle gerbes carried out in \cite[\S 7.2.]{NW}, compare the definition of the groupoid composition and the  Lie $2$-group  action with  \cite[Formula (7.2.1)]{NW}.}
\end{proof}
{ This yields a functor $\Psi$ which sends an $(H\rtimes G)$-bundle gerbe morphism $f$ to the
	$(H\rtimes G)$-groupoid morphism ${\rm Id}_G\times f$.}

\begin{rk}\label{rk:Psifunctor}
 	Note that the image of  $\Psi$ restricted to the bundle gerbes over fiber product groupoids $Y^{[2]}$, is a subset of the principal bundle groupoids over fiber product groupoids with trivial base. 
 \end{rk}
\subsection{From a base trivial PB groupoid to a bundle gerbe  over  a groupoid} 
{On the grounds of Proposition \ref{prop:grb.to.pb}, we specialise to the class of trivial base  principal bundle groupoids.}
{The  functor $\Phi$ built above induces a functor:
$$\Xi:\text{ trivial base $(H\rtimes G)$- PB  groupoid over $Y^{(1)}$ }\fto \text{$(H\rtimes G)$-bundle gerbes over}  Y^{(1)}$$}

Let $H\rtimes G \curvearrowright {P^{(1)}}\to Y^{(1)}$ be any base trivial PB groupoid.  {Let ${\rm pr}_G\colon P= G\times Y \to G$ be the canonical projection onto the first component and consider the map }
\[\bs_G:={  {\rm pr}_G\circ \bs} \colon P^{(1)}\fto G . \]

By the implicit function theorem, the preimage of the neutral element $e_G\in G$ given by $B:=\{ \xi\in P^{(1)} \st \bs_G(\xi)=e_G\}$ is a submanifold of $P^{(1)}$. Recall that  $G^{(1)} =	H\rtimes G$ acts on $P^{(1)}$ over the trivial $G$-action  on $P=G\times Y$. We have the following diffeomorphism:

\begin{equation}\label{eq:varphi}
\begin{matrix}
\varphi:	P^{(1)} &\fto &  {G\times B}\\
\xi& \mapsto& {\left(\bs_G(\xi),\bs_G(\xi)^{-1}\,\xi \right)},
\end{matrix} 
\end{equation}
{ (where $g\, \xi= (e, g)\cdot \xi$),} corresponding to the diagramme:
\[\begin{tikzcd}
& G\times B & \\
{	H\rtimes G} \ar[d, shift right=.2em, swap,"\bt "]\ar[d,shift left=.2em,"\bs"] \ar[bend left=25, shift right=.8em,r] & {P^{(1)}} \ar[u, "\varphi"] \ar[d, shift right=.2em, swap,"\bt"]\ar[d,shift left=.2em,"\bs"] \ar[r] & Y^{(1)} \ar[d, shift right=.2em, swap,"\bt"]\ar[d,shift left=.2em,"\bs"]\\
G \ar[bend left=25, shift right=.8em,r] & {G\times Y } \ar[r] & Y
\end{tikzcd}.\]	

{ In the particular case $Y^{(1)}=Y^{[2]}$,} the construction of the previous section yields
  an $(H\rtimes G)$-bundle gerbe $\Psi(P^{(1)})\overset{\varphi}{\simeq} G\times G\times B $ over $P^{[2]}\simeq G\times G\times Y^{[2]}$.  We want to interpret $B$ as a bundle gerbe  over $ Y^{[2]}$.
  Going back to the case  $P^{(1)}\to Y^{(1)}$, let  us first note that the following diagram commutes:
  $$\begin{tikzcd}
  G\times B \ar[r, "\varphi"] \ar[swap,shift right=.2em,d,"{\rm Pr}_B"] & P^{(1)} \ar[d,"\pi"] \ar[l] \\
  B \ar[swap, shift right=.2em,u,"\{e_G\}\times {\rm Id}_B"] \ar[r,"\Pi=\pi\circ \varphi^{-1}|_B"] & Y^{(1)} 
  \end{tikzcd}$$
  Since $\pi$ is a surjective submersion, so is  $\Pi:= (\pi\circ \varphi^{-1})\vert B$ a surjective submersion.



To prove that $\Pi\colon B\to Y^{(1)}$  is an $(H\rtimes G)$-bundle gerbe, we follow  the construction in \cite[\S 7.1]{NW}, where the authors build a bundle gerbe  from a PB groupoid, aswell as the construction in \cite[\S 7.2]{NW}, where the authors build a  PB groupoid from a bundle gerbe, see in particular Formula (7.2.1).\\


\begin{prop} \label{prop:btpb.to.grb} Let $\pi\colon P^{(1)}\to Y^{(1)}$ be a trivial-base $H\rtimes G$- PB groupoid with $P^{(1)}\soutar P$ such that $P\simeq G\times Y$. The submanifold $\Xi(P^{(1)}):=B={\rm Ker}(\bs_G)$ is a $(H\rtimes G)$-bundle gerbe 
	\begin{equation}\label{eq:diagrammtrivialbase}\begin{tikzcd}
	H\rtimes G \ar[d, shift left=.2em,"\bs"]\ar[d, shift right=.2em, swap,"\bt"] \arrow[r, bend left=30,swap, "\star"] &   B  \ar[dl, "\rho"] \ar[r, "\Pi "]&    Y^{(1)} \ar[d, shift left=.2em,"\bs"] \ar[d, shift right=.2em, swap,"\bt"]&  & \\
	{G} & &   Y    \ar[r,"\pi"] & M\\
	\end{tikzcd}. \end{equation}
	over the groupoid $Y^{(1)}$ with:
	\begin{itemize}
		\item { projection  $\Pi:=(\pi \varphi^{-1})\vert_B\colon B\to Y^{(1)}$, }
		\item anchor $\rho=(t_G)^{-1}:=(-)^{-1}\circ {\rm pr}_G\circ ({\bt}|_B)\colon B\fto G;b\mapsto ({\rm pr}_G( \bt(b)))^{-1}$,
		\item the ${H\rtimes G}\soutar G$ action given in terms of the groupoid action:
		$$( h,\rho(b))\star b = (C_{\rho(b)^{-1}}h^{-1},e)\cdot b,
		$${ where $b$ is viewed as an element of $P^{(1)}$.}
		\item bundle gerbe product in terms of the groupoid composition $\circ_{P^{(1)}}$:
		\begin{equation}
			\begin{matrix}
				\circ_\mu\colon & B_{\bs\Pi}\!\times_{\bt\Pi} B_{\gamma_1}&\fto&\hspace{.7in} B \hspace{.7in}\\
				 & (b_2,b_1)&\mapsto& ((e,\rho(b_1)^{-1})\cdot b_2)\circ_{P^{(1)}} b_1 \,\, ,
			\end{matrix}
		\end{equation}
		{ where $b_2$ is viewed as an element of $P^{(1)}$.}
	 
	\end{itemize}
\end{prop}
\begin{rk}{ Note that   the above formula for the action $\star$,  resp. the bundle gerbe product $\circ_\mu$   follows from setting $\phi=b$, $g=\rho(b)$, resp. $g_1=\rho(b_1)$  in the formulas for the action, resp. the bundle gerbe product of  Proposition  \ref{thm:pb.to.grb}, which we recall here for convenience:
		$$(h,g)\star (g,\phi):=\left(\,d(h)g \,\, , \,(C_{g^{-1}}h^{-1},e)\cdot \phi \,\right), $$   resp.
	$$(g_2,\phi_2)\circ_\mu (g_1,\phi_1):=\left(\, g_2 \cdot g_1 \,\, , \,\, \left( (e,g_1^{-1})\cdot \phi_2\right) \circ_{P^{(1)}} \phi_1 \,\, \right),$$ and then applying the projection ${\rm pr}_B$.}
\end{rk}

\begin{proof}{ We know that    $\Pi:= (\pi\circ \varphi^{-1})\vert B$ is a surjective submersion. The map  $\varphi$ is a $H\rtimes G$-equivariant isomorphism where $H\rtimes G$ acts on $G\times B$ by a trivial action on $G$ and the groupoid action on $B$ and  $B\to Y^{(1)}$ is a $H\rtimes G$-principal bundle.} 
			
			{As a consequence of the above remark, the proof of the properties of the various operators is similar to that of Proposition \ref{thm:pb.to.grb}:}  
			\begin{itemize}  
		\item $\star$ is a Lie groupoid action  since $\cdot$ is a smooth action. It acts freely since $\cdot$ acts freely.
		\item $\circ_\mu$ is smooth and associative since  $\cdot$ and $\circ$ are smooth and associative.
		\item The bundle gerbe product is $(H\rtimes G)$-equivariant following the same argument used in the proof of Proposition \ref{thm:pb.to.grb}. 
		
	\end{itemize}
\end{proof}
{This gives a functor $\Xi$ from $(H\rtimes G)$-PG groupoids to $(H\rtimes G)$-bundle gerbes by restriction of the PG groupoid morphisms to $B={\rm Ker}(\bf s_G)$.}

\subsection{Conclusions}

Let $\Phi$ be the correspondence of Proposition \ref{thm:pb.to.grb}, $\Psi$ be the correspondence of Proposition \ref{prop:grb.to.pb} and  $\Xi$ be the correspondence of Proposition \ref{prop:btpb.to.grb}. For a fixed Lie 2-group $H\rtimes G$, the correspondences $\Phi,\Psi $ and $ \Xi$ are functors between the (strict)-categories of $H\rtimes G$-principal bundle groupoids and $H\rtimes G$-bundle gerbes.

The correspondences $\Xi$ and $\Psi$ are functors between the categories of base trivial principal bundles with strict morphisms and gerbes with (strict) morphisms. Moreover, they are inverses of each other.

 For a fixed Lie 2-group $H\rtimes G$, the correspondences $\Psi$ and $\Phi$ are functors between the (strict)-categories of $H\rtimes G$-principal bundle groupoids over fiber product groupoids and $H\rtimes G$-bundle gerbes over fiber product groupoids. Although $\Psi$ and $\Phi$ are not inverses of each other, in \cite{NW} it is shown that they give an equivalence of categories, when seen   as $2$-categories with weak equivalences and anafunctors.

{The following statement follows from the above constructions summarised in the diagramms (\ref{eq:diagrammfibration})  and (\ref{eq:diagrammtrivialbase}). 
	
\begin{prop}\label{prop:bg.pq} Let $H\rtimes G$ be a Lie 2-group and $P^{(1)}\fto{ Y^{(1)}}$  a principal $(H\rtimes G)$-principal bundle groupoid. The bundle gerbes $\Phi(P^{(1)})$ and $\Xi(P^{(1)})$ are   equivalent  to the partial quotient $P^{(1)}/_G$ of Definition \ref{def:gauge.grpd}. { To show this we use Definition\ref{defn:Mor2}:} \begin{enumerate}
 	\item If $Y^{(1)}=Y^{[2]}$ then  the bundle gerbe $B:=\Phi(P^{(1)}){=G\times P^{(1)}}$ is isomorphic (as a groupoid) to the pullback by $Q:P\fto Y$ of the partial quotient $P^{(1)}/_G$; this is
 	$$B\cong P \,_Q \! \times_\bt (P^{(1)}/_G) \,_\bs \! \times_Q P,$$
 	with the isomorphism given by $(g,\gamma)\mapsto (g\,\bt(\gamma),[\gamma],\bs(\gamma))$ for any $g\in G$ and $\gamma\in P^{(1)}$.
 	\item
 	If  $P\simeq G\times Y$, so if  $P^{(1)}\to Y^{(1)}$ is a trivial base  PB groupoid, the bundle gerbe $B':=\Xi(P^{(1)})=   P^{(1)}\vert {{\rm Ker}(\bs_G)}$ is isomorphic to the partial quotient $P^{(1)}/_G$.
 	\end{enumerate}\end{prop}}

\section{Nerve of a principal bundle groupoid}\label{sec:PBnerve}
{  This section  discusses  higher analogs of the principal bundle groupoids introduced in the previous one. We first review  some   concepts used to build these higher analogs.}
\subsection{Nerve of a small category}

Let $\mathbf \Delta$ denote the category whose objects are finite ordered sets $[n]=\{0,1,2,\cdots, n\}$ and whose morphisms are order-preserving functions $[m]\to [n]$. One defines the injection   $d_i\colon [n-1]\to [n]$ omitting $i\in [n]$ and the surjection $e_i:[n]\to [n-1]$  repeating $i\in [n-1]$.

A {\bf simplicial   set}  in a category $\mathbf C$ is  a  contravariant  functor $X\colon {\bf  \Delta} \to \mathbf C$  (equivalently,  a  covariant  functor  $X\colon {\bf  \Delta}^{op} \to \mathbf C$). An element $x_k\in  X^{(k)}:=X([k]) $ is called an {\bf  $k$-simplex}.   A  morphism  of  simplicial objects in ${\bf C}$ is a natural transformation of such functors.

Concretely, a simplicial object in the category ${\mathbf C}$  amounts to a sequence $C^{(k)}, k\in \mathbb Z_{\geq 0}$ of objects in ${\mathbf C}$ called {\bf simplices} and a collection of morphisms called {\bf face maps} $d_k\colon C^{(k)}\to C^{(k-1)}$ and {\bf degeneracy maps} $e_k\colon C^{(k-1)}\to C^{(k)}$  which obey the following conditions (see e.g. \cite{F})
\begin{eqnarray*}
 d_i\, d_j= d_{j-1}\, d_i \quad {\rm for}\, i<j \,\, ,& &d_i\, e_j= e_{j}\, d_{i-1} \quad {\rm for}\, i>j+1 \,\, , \\
	e_i\, e_j= e_{j+1}\, e_i   \quad {\rm for}\, i> j\,\, ,& & 	d_j\, e_j= d_{j+1}\, e_j={\rm Id}\,\,,\\
d_i\, e_j= s_{j-1}\, d_i \quad {\rm for}\, i<j \,\,. & & 
\end{eqnarray*}

Let $\mathcal{C}$ be a category. The {\bf nerve} $N^{\bullet}( M^{(1)})=M^{\bullet}$ of a small category $M^{(1)}$ in $\mathcal C$ is a simplicial set (in the category of sets). The $0$-simplices $M^{(0)}$ are given by the points in $M$, the $1$-simplices  by the arrows $M^{(1)}$, its $k$- simplices are the sets of $k$ arrows morphisms in $M^{(1)}$, i.e. they are the sets
$$M^{(k)}=N^k(M^{(1)}):=M^{(1)}\!_s \!\times_t \overset{k}{\dots} _s \!\times_t M^{(1)} = \left\{ (\alpha_1, \cdots, \alpha_k)\in (M^{(1)})^k \st \alpha_i \text{ is composable with } \alpha_{i+1}\right\}.$$ 

For $k=1$, $d_0$, resp. $d_1$ is the source, resp. the target map from the arrows to the points of the category.  For $k\geq 2$, the two outer face maps $d_0$, respectively $d_{k+1}$, are defined by forgetting the first, respectively the last morphism in such a sequence. The $k$ inner face maps $d_j, j=1, \cdots, k$  are given by composing the $j$-th morphism with the $j+1$-st morphism in the sequence. The degeneracy maps $e_j, j=0, \cdots, k+1$ are given by inserting an identity morphism on $x_j$.

\begin{rk}
	The nerve alone contains all the information needed to reconstruct the small category $M^{(1)}$.  If the big category $\mathcal C$ is closed by fiber products then $M^{\bullet}$ is a simplicial set in $\mathcal C$.
\end{rk}  

{\bf  Nerves of  groupoids}, which we shall consider in this paper, are in one to one correspondence with Kan complexes. These are simplicial sets $X$ with the property that, for   $0< k\leq n$, any morphism of simplicial sets $  \Delta_k^n\to X$, where $\Delta_k^n$ is obtained by removing the interior of $\Delta^n $ and that of the faces $d_k \Delta^n$,   can be extended to  simplicial morphism $\Delta^n \to X$. 

\subsection{The nerve of a Lie $2$-group }

Associated to any Lie $2$-group $G^{(1)}\soutar G$, there is a simplicial manifold $N^{  \bullet } ({G^{(1)}})=G^{\bullet}$. It  is a simplicial set in the category of Lie groups i.e.,  the $k$-th nerve $N^{k}({G})=G^{(k)}=G^{(1)} \!_s \!\times_t \overset{k}{\dots} _s \!\times_t G^{(1)}$  is endowed with  the structure of a Lie group via the  coordinate-wise multiplication, and the faces and degeneracy maps are Lie group  morphisms. The nerve of a Lie-2 group $G^{(1)}$ is a simplicial set on the category of Lie groups.

We shall give a few properties of the nerve of a Lie 2-group which even if elementary will be useful for the sequel. 

The Lie $2$-group $G^{(1)}=H\rtimes G$ defines other two distinct simplicial sets in the category of Lie groups:
\begin{enumerate}
	\item The family of Lie groups given by $\left\{H\rtimes\left(\overset{k-1}{\dots} \rtimes (H\rtimes G)\right)\right\}_{k\geq 0}$ together with face and degeneracy maps given by:
	\begin{itemize}
		\item $\delta_i^k(h_k,\dots, h_1,g)=(h_k,\dots,h_{i+1}h_{i},\dots,h_1,g)\in H\rtimes\left(\overset{k-1}{\dots} \rtimes (H\rtimes G)\right)$,
		\item $e_j^k(h_k,\dots, h_1,g)=(h_k,\dots,h_{j+1}, e_H,h_{j},\dots, h_1,g)\in H\rtimes\left(\overset{k-1}{\dots} \rtimes (H\rtimes G)\right)$,
	\end{itemize}
	for all $0<i< k$ and $0\leq j\leq k$.
	\item The family of Lie groups given by $\{ H^k\rtimes G\}_{k\geq 0}$ together with:
	\begin{itemize}
		\item Face maps: $\delta_i^k({ \bf h}_k,.., { \bf h}_1,g)=({ \bf h}_k,..,\widehat{ \bf h}_{i},\dots,{\bf  h}_1,g)\in H^{k-1}\rtimes G$,
		\item Degeneracy maps: $e_j^k({\bf h}_k,.., {\bf  h}_1,g)=({\bf  h}_k,..,{\bf  h}_j, { \bf h}_j,.., { \bf h}_1,g)\in  H^{k+1}\rtimes G$,
	\end{itemize}
	for all $0<i< k$ and $0\leq j\leq k$.
\end{enumerate}

These three simplicial sets in the category of Lie groups are isomorphic.

\begin{prop}\label{prop:nerve.2-grp}
	The following maps define an isomorphism of simplicial sets in the category of Lie groups for the Lie-2 group $G^{(1)}\cong H\rtimes G$.
	\begin{eqnarray}
		G^{(k)}\quad \quad \quad\quad & \fto &  H\rtimes\left(\overset{k-1}{\dots} \rtimes (H\rtimes G)\right) \\
		(h_k, t(h_{k-1}\cdots h_1)g)\times\cdots\times (h_1,g) &\mapsto& (h_k,\dots ,h_1,g). \nonumber
	\end{eqnarray}
	\begin{eqnarray}
		H\rtimes\left(\overset{k-1}{\dots} \rtimes (H\rtimes G)\right) & \fto & H^k \rtimes G\\
		(h_k,..,h_1,g) &\mapsto & ({ \bf h}_k,..,{ \bf h}_2,{ \bf h}_1,g):=(h_k\cdots h_1,..,h_2 h_1,h_1,g). \nonumber
	\end{eqnarray}
\end{prop}

\subsection{The nerve of a PB groupoid and of its partial quotient groupoid}

{The following statement is motivated  by  Propositions  \ref{prop:lie2grp.q}.} and \ref{prop:PB groupoidcat}.
\begin{thm} \label{prop:princ.2.bnd}
	{ The nerve of a principal bundle groupoid is a collection of principal bundles given by the nerves of the groupoids:
		\[\left[P^{(1)}\underset{ G^{(1)}}{ \longrightarrow} M^{(1)}\right]\rightsquigarrow \left[P^{\bullet}\underset{G^{\bullet}}{ \longrightarrow} M^{\bullet}\right]. \]}
	
	{Moreover}, a principal (${G^{(1)}} \soutar G$)- bundle groupoid ($P^{(1)}\soutar P$) over ($M^{(1)}\soutar M$) defines a simplicial set in the $(1)$-principal bundle category, given by ${G^{\bullet}}$ acting on $P^{\bullet}$ with quotient space being $M^{\bullet}$.	
\end{thm}
\begin{proof} By functoriality, this follows from  the fact that the action map $G^{(1)}\times P^{(1)}\fto P^{(1)}\fto M^{(1)}$ is a Lie groupoid morphism. Nonetheless, we  give a pedestrian proof of the statement.\\
	{\bf Simplicial set:} 
	We note  that the action map extends to a map of simplicial sets:
	$$(G^{(1)}\times P^{(1)} )^{(k)}\fto P^{(k)}\fto M^{(k)}\quad \forall k\in \mathbb N,$$
	Moreover, $( G^{(1)}\times P^{(1)})^{(k)}\cong G^{(k)}\times P^{(k)}$, which yields a map of simplicial sets:
	$$G^{(k)}\times P^{(k)} \fto P^{(k)}\fto M^{(k)} \quad \forall k\in \mathbb N.$$
	
	\noindent{\bf Simplicial set on Principal bundles:} Since the action map $G^{(1)}\times P^{(1)}\fto P^{(1)}$ is a group action, the maps given by the cartesian product $(G^{(1)})^{k}\times(P^{(1)})^{k}\fto (P^{(1)})^{k}$ are also group actions. From which it follows that the restriction to the submanifolds $ G^{(k)}\times P^{(k)}\fto P^{(k)}$ are also group actions $ \forall k\in \mathbb N $.
	
	Moreover, $(G^{(1)})^k\times (P^{(1)})^k\fto (P^{(1)})^k\fto (M^{(1)})^k$ are principal bundles and hence so do the following submanidols build   principal bundles:
	$$G^{(k)}\times P^{(k)}\fto P^{(k)}\fto M^{(k)},\quad \forall k\in \mathbb N.$$
	
	For a groupoid $\CA\in \{G^{(1)},P^{(1)},M^{(1)}\}$, let $d_i^\CA:\CA^{(k)}\rightarrow \CA^{(k-1)}$ and $e_i^{\CA}:\CA^{(k-1)}\rightarrow \CA^{(k)}$ denote the face and degeneracy maps respectively. The maps $d_i^G\times d_i^P\colon  G^{(k)}\times P^{(k)}\mapsto G^{(k-1)}\times P^{(k-1)}$ and $e_i^G\times e_i^P\colon G^{(k-1)}\times P^{(k-1)} \mapsto G^{(k)}\times P^{(k)} $ are the face and degeneracy maps of the cartesian product groupoid $P^{(1)}\times G^{(1)}$. Consider the following   diagrammes  here combined in one sole diagramme, each of which involves   the face maps $d_i$, resp. the degeneracy maps $e_i$: \\
	$$\begin{tikzcd}
		G^{(k)}\times P^{(k)}  \ar[r]\ar[d,shift right=.2em, swap,"d_i^P\times d_i^G"] & P^{(k)} \ar[r] \ar[d,shift right=.2em, swap, "d_i^P"] & M^{(k)} \ar[d,shift right=.2em, swap, "d_i^M"] \\
		G^{(k-1)}\times P^{(k-1)}\ar[r] \ar[u,shift right=.2em, swap,"e_i^P\times e_i^G"] & P^{(k-1)} \ar[r] \ar[u,shift right=.2em, swap,"e_i^P"]& M^{(k-1)} \ar[u,shift right=.2em, swap,"e_i^M"]
	\end{tikzcd} \hspace{.2in} \forall k\in \KN.$$
	They commute in so far as the  the face maps $d_i$, resp. the degeneracy maps $e_i$, commute with the horizontal ones since the horizontal arrows are simplicial sets morphisms coming from a groupoid morphism. Therefore, the maps $(d_i^G,d_i^P,d_i^M)$ and $(e_i^G,e_i^P,e_i^M)$ are principal bundle morphisms.
\end{proof} 

{The subsequent statement follows from Proposition  \ref{prop:gaugeprincipal} combined with Theorem \ref{prop:princ.2.bnd}.}
\begin{prop} \label{prop:NerveG}   With the notations of Proposition \ref{prop:gaugeprincipal}, the projection $N^\bullet(Q)\colon N^\bullet\left({P^{(1)}}\right)\longrightarrow N^\bullet\left({ {P^{(1)}}/_G}\right)$ induced by $Q$ is a simplicial set of principal $G$-bundles and $N^\bullet(Q)$ induces the following isomorphisms for any $k\in \mathbb N$:
	
	\begin{eqnarray}\label{eq:GkGk-11}  	\tilde{N}^k(Q):  \left( N^k( { {P^{(1)}}})\right)  /_{G} \   &\overset{\simeq}{\longrightarrow} &  N^k\left( { {P^{(1)}}/_G}\right) \nonumber \\
		\left[\phi_1,   \cdots, \phi_k\right] &\longmapsto&
		\left([\phi_1],[\phi_2],\cdots, [\phi_k]\right).
	\end{eqnarray}
	Moreover, the face maps on the nerve of ${P^{(1)}}$	give rise to face maps $\partial_i^k$ for any $k\in \mathbb N$:
	
	\begin{eqnarray}\label{eq:faces}\partial_i^k:  N^k\left({P^{(1)}}\right) /_{ {G}} &\longrightarrow &  N^{k-1}\left( {P^{(1)}}\right)  /_{ {G}}\nonumber \\
		\left[\phi_1,   \cdots, \phi_k \right] &\longmapsto&   \left[\phi_1,\cdots,\phi_i\circ \phi_{i+1}, \cdots, \phi_{k+1}\right]. 
	\end{eqnarray}

	Denoting by $\delta_i^k$ the face maps of the groupoid $G^{(k)}$, the subsequent diagramme commutes: 
	\[\begin{tikzcd}
		N^{k}\left( { {P^{(1)}}/_G} \right)\arrow[r, "\delta_i^k"] &    N^{k-1}\left( { {P^{(1)}}/_G}\right)   \nonumber \\
		N^{k}\left( {P^{(1)}}\right)/_{G} \arrow[r, "\partial_i^k"]  \arrow[u, " N^ {k}(Q)"]&  N^{k-1}\left( {P^{(1)}}\right)/_{ G} \arrow[u,swap, " N^ {k-1}(Q)"]
	\end{tikzcd}\]
\end{prop} 

We apply the above constructions to  the pair groupoid of a principal bundle.

\begin{ex} Let $P\fto M$ be a  principal $G$-bundle and $P^{(1)}=\CP^{(1)}(P)$, $M^{(1)}=\CP^{(1)}(M)$ be the corresponding pair groupoids. We saw that in this case,  $P^{(1)}/G$ is   the gauge groupoid $\CG(P)$.
	
	The isomorphisms of Proposition \ref{prop:NerveG}   are the following maps:
	\begin{eqnarray}\label{eq:GkGk-1}  	\tilde{N}^k (Q):  {\CP}^{(k)}(P)   /_{  G} \   &\overset{\simeq}{\longrightarrow} & {\CG}^{(k)}(P) \nonumber \\
		{\bf [p]}:=	[p_1,   \cdots, p_{k+1} ]&\longmapsto&  ([p_1,p_2],[p_2,p_3],\cdots, [p_k,p_{k+1}]).
	\end{eqnarray}
	
	Moreover, the face maps on the nerve of the pair groupoid $\CP({P})$
	give rise to face maps $\partial_i^k$:
	\begin{eqnarray}\label{eq:faces}\partial_i^k:{\CP}^{(k)}(P)   /_{G} &\longrightarrow & {\CP^{(k)}(P)}   /_{ G}\nonumber \\
		\left[p_1,   \cdots, p_{k+1} \right]&\longmapsto&   \left[p_1,\cdots,\widehat{p_i}, \cdots, p_{k+1}\right]. 
	\end{eqnarray}  
	Denoting by $\delta_i^k$ the face maps of the partial quotient groupoid $G^{(k)}(P)$, the subsequent diagramme commutes: 
	\[\begin{tikzcd}
		{\CP}^{(k)}(P)/_{ G} \arrow[r, "\de_i^k"] \arrow[d, " Q_k"] &   {\CP}^{(k-1)}(P)/_{ G} \arrow[d, " Q_{k-1}"]\nonumber \\
		{\CG}^{(k)}(P) \arrow[r, "\delta_i^k"] & {\CG}^{(k)}(P)\\
	\end{tikzcd}\]
\end{ex}
	\section{Inner transformations of a PB groupoid} \label{ec:innertrans}
	 
	In this section we investigate gauge transformations of  principal bundle groupoids and specialise to the case of the pair groupoid of a principal bundle manifold.
	
With the notations of Definition  \ref{defn:Lie3pb},  let 
  $P^{(1)}\soutar P  $ be a $H\rtimes G$-principal bundle groupoid over a Lie groupoid $M^{(1)}\soutar M$.

\subsection{Inner transformations of a PB groupoid}\label{sec:inner}

Let us recall that the gauge transformations of   the principal $G$-bundle    $\pi\colon P\rightarrow M$ is the group of principal bundle morphisms over the identity in $M$, i.e.
$$\mathrm{Mor}(P):= \{\phi\in \CI(P,P)\st \pi(p)=\pi(\phi(p))\, ; \,\, \phi(g\, p)=g\, \phi(p)\,\,\, \forall \, g\in G \,,\, p\in P\}.$$

This is a group because there is a well known isomorphism: 
\[ {\rm Mor}(P)\simeq C^\infty_G(P,G),\]
where the multiplication in ${\rm Mor}(P)$ is the composition of morphisms; the set on the r.h.s. are $G$-equivariant smooth maps from $P$ to $G$ with the group multiplication, and $G$ acts by the left action on $P$ and by the (left) adjoint action on $G$. \\

{\it Here and in what follows Aut$(P)$ stands for invertible elements in Mor$(P)$; the mention of a group in the lower index such as $G$ in $C^{\infty}_G$, refers to equivariance maps with respect to the action of the group $G$. A priori ${\rm Aut}(P)\subset {\rm Mor}(P)$, but the isomorphism ${\rm Mor}(P)\simeq C^{\infty}_G(P,G)$ leads to the conclusion ${\rm Aut}(P)={\rm Mor}(P)$. Therefore, in the rest of this paper ${\rm Aut}(-)$ and ${\rm Mor}(-)$ are interchangeable.} \\

Propositions \ref{prop:princ.2.bnd} and \ref{prop:nerve.2-grp} give rise to a simplicial set of  principal bundles whose $k$-nerves are principal $ H^k \rtimes G$-bundle groupoids $P^{(k)}\fto M^{(k)}, k\in \mathbb N$ and  bijections: 	
\begin{equation}\label{eq:Phik}\Psi_k\colon {\rm Aut}(P^{(k)})\overset{\simeq}{\longrightarrow} C^{\infty}_{H^k\rtimes G}(P^{(k)}, H^k\rtimes G), 
\end{equation}

which in turn give rise to group isomorphisms. In a compact form, this reads
\begin{equation}\label{eq:AutNP} {\rm Aut}(P^\bullet)\cong C^{\infty}_{H^\bullet\rtimes G}(P^\bullet,H^\bullet\rtimes G).
\end{equation}
and we refer to ${\rm Aut}(P^\bullet)$ as the group of  {\bf inner transformations of the  nerve of  the   principal bundle groupoid $P^{(1)}$}, or for short the {  \bf inner transformation group of   $P^{\bullet}$}.
\subsection{Inner transformations of the partial quotient of a PB groupoid}\label{sec:innerpartial}

Given a PB-groupoid $H\rtimes G \curvearrowright P^{(1)}\overset{\pi^{(1)}}{\longrightarrow} M^{(1)}$, let $P^{(1)}/_G$ be the partial quotient (as in \ref{eq:GkGk-11}) and let us abuse notation to call the map given by the original PB-groupoid $\pi^{(1)}\colon P^{(1)}/_G{\longrightarrow} M^{(1)}$ again. This is not a $PB$-groupoid, in general there is no Lie 2-group acting on $P^{(1)}/_G$.  Let us denote following sets and maps for every $k\in \mathbb N$:
$${\CI}(P^{(k)}/_G, {P}^{(k)}/_G)^{M^{(k)}}:=\left\{\, \varphi\in \CI(P^{(k)}/_G,P^{(k)}/_G) \st \varphi\circ\pi^{(k)}=\varphi \, \right\},$$
$$  \Pi     _k\colon {\rm Aut}( P^{(k)})\longrightarrow  C^\infty ({P}^{(k)}/_G, {P}^{(k)}/_G)^{M^{(k)}}$$
$$({\vec \phi}\mapsto \varphi({\vec \phi})=\left( \varphi_0({\vec \phi}), \cdots, \varphi_{k+1}({\vec \phi})\right))\longmapsto\left([{\vec \phi}]\mapsto[\varphi({\vec \phi})]=\left[\left( \varphi_0({\vec \phi}), \cdots, \varphi_{k+1}({\vec \phi})\right)\right]\right),$$
\begin{eqnarray} \label{eq:Gammak} \Gamma_k\colon \CI(P^{(k)},H^k\rtimes G)&\longrightarrow &\CI(P^{(k)}, H^k)\nonumber\\
	(h_k, \cdots, h_1,g_0)&\longmapsto&  \left(  
	C_{g_0}h_k,\dots,C_{g_0}h_{1} \right).
\end{eqnarray} 
The maps $\Gamma_k$ clearly preserves equivariant properties w.r.to the partial action of $H^k\rtimes G$ on $P^{(k)}$ and the adjoint action of $H^k\rtimes G$  on $H^k$ by the adjoint action. Therefore, let us consider the maps
\begin{eqnarray}\label{eq:Xik}
	\Xi_k\colon C^{\infty}_G(P^{(k)}, H^{k})&\longrightarrow &\CI({ {P}^{(k)}/_G}, {P}^{(k)}/_G)^{M^{(k)}}\nonumber\\
	({\vec \phi}\mapsto (h_k({\vec \phi}), \cdots, h_1({\vec \phi}))&\longmapsto&  \left([{\vec\phi}]\mapsto \left[  (h_k({\vec\phi}), \cdots, h_1({\vec\phi}),1)\cdot {{\vec\phi}}\,\right]\right),
\end{eqnarray} 
which can easily seen to be   injective.\\

\label{thm:map.kPG2}
The following diagramme commutes: 

\begin{equation}\label{eq:barPhik}
\begin{tikzcd}
\CI_G  (P^{(k)}, H^k) \arrow[d, "    \Xi_k  "] &     \CI_{H^k\rtimes G} (P^{(k)}, H^k\rtimes G)    \arrow[l,swap, " \Gamma_k"]    \\
\CI({ {P}^{(k)}/_G},{ {P}^{(k)}/_G})^{M^{(k)}} &   {\rm Aut} ({P}^{(k)}) \arrow[u, " \overline \Psi_k"] \arrow[l, "   \Pi_k"],  
\end{tikzcd}
\end{equation}
		where  $\overline \Psi_k$ is the map induced by (\ref{eq:Phik}).\\
		
		This yields the following bijections: 
		\[ 
 {\rm Aut}({P}^{(k)})/_{\sim_{\Pi_k}} \cong \, \CI_{H^k\rtimes G} ({P}^{(k)},H^k\rtimes G) /\CI_{H^k\rtimes G} ({P}^{(k)}, \{e\}^k\rtimes G) \,
		\cong  \CI_{H^k\rtimes G} ({P}^{(k)}, H^k)\]

		where $\sim_{\Pi_k}$ is the equivalence relation with classes equal to the fibers of $\Pi_k$. These bijections are isomorphisms of groups.	On the grounds of these identifications, we set  the following definition.
	\begin{defn} \label{defn:MorNG}For any $k\in \mathbb N$, we call  elements of  the sets 
		\begin{equation}\label{eq:MorGk}{\rm  Aut}({ {P}^{(k)}/_G}):= {\rm Aut} ({P}^{(k)})/_{\sim_{\Pi_k}}.
		 \end{equation}
	{\bf  	inner transformations } of $ {P}^{(k)}/_G $ and the resulting family ${\rm  Aut}({ {P^\bullet}/_G}):=   {\rm Aut} (P^\bullet )/_{\sim_{\Pi_\bullet}}$ 
	the  	{\bf  	inner transformations  of} $  { {P^\bullet}/_G} $ i.e. the nerve of the partial quotient groupoid ${ {P^{(1)}}/_G}$.
\end{defn}

\begin{thm}\label{thm:MorNGP2} For any $k\in \mathbb N$, there is a group isomorphism:
		\[{\rm Aut}({ {P}^{(k)}/_G})\simeq 
		\CI_{H^k\rtimes G}({P}^{(k)}, H^k).  \] 
			
		 Correspondingly, inner transformations  of ${P^{\bullet}}$ form a group with the following group isomorphism:
		  \begin{equation}\label{eq:MorNG} {\rm  Aut}({ {P^\bullet} /_G})={\rm Aut}(P^\bullet )/_{\sim_{\Pi_\bullet}}\simeq C^{\infty}_{H^\bullet\rtimes G} ({P^\bullet}, H^\bullet).
		  \end{equation}
\end{thm}

	  We now consider a bundle gerbe $B^{(1)}$ and, as before, denote by $\Psi$ the functor of Proposition \ref{prop:grb.to.pb}. Then $P^{(1)}=\Psi(B^{(1)})$ is a base trivial PB groupoid, $B=\ker(\bs_G)\overset{\pi^{(1)}}{\longrightarrow} M^{(1)}$ as in Prop \ref{prop:btpb.to.grb}; also $B^{(k)}\simeq P^{(k)}/_G$ by Proposition \ref{prop:bg.pq}. Since the functor $\Psi$ of Proposition \ref{prop:grb.to.pb} and the functor $\Xi$ of Proposition \ref{prop:btpb.to.grb} are inverses of each other we get that 
		$${\rm Aut}(B^{(1)}):=\left\{\, \varphi\in \CI_{H\rtimes G}(B^{(1)},B^{(1)}) \st \varphi\circ\pi^{(1)}=\varphi \,\right\}\simeq {\rm Aut}(P^{(1)})|_{B^{(1)}}\simeq {\rm  Aut}(P^{(1)}/_G)$$
		   so that the following statement is a straightforward consequence of Theorem \ref{thm:MorNGP2}.
	 \begin{cor}\label{cor:AutBk} { For any $k\in \mathbb N$, there is a isomorphism:
	 	\[{\rm  Aut}(B^{(k)})={\rm Mor}(B^{(k)})\simeq	{\rm  Aut}(P^{(k)}/_G)\simeq 
	 	\CI_{H^k\rtimes G}({P}^{(k)}, H^k).  \] 
	 	
	 Correspondingly, inner transformations  of the nerve 
	 	of ${B^{\bullet}}$ form a group with the following group isomorphism:
	 	\begin{equation} {\rm  Aut}({B^\bullet})\simeq{\rm  Aut}(P^{\bullet}/_G)\simeq \CI_{H^\bullet\rtimes G} ({P^\bullet}, H^\bullet).
	 	\end{equation}}
	 \end{cor} 
	 
	\subsection{The case of the pair groupoid of a principal bundle }

Specialising    Proposition \ref{prop:gaugeprincipal} to the pair groupoid ${P^{(1)}}={\CP^{(1)}}(P)$,  we consider the canonical projection of the principal bundle $G\curvearrowright P\to M$,
\[  Q\colon {P^{(1)}}\longrightarrow {\CG}(P):=P^{(1)}/_G\]
onto the quotient partial quotient groupoid ${\CG}(P)$ by the diagonal action. Applying Theorem \ref{thm:MorNGP2} (with $H=G$) we get the following corollary.

\begin{cor}\label{cor:GaugeNGP}
Gauge transformations of the nerve of  the partial quotient groupoid $\CG(P)$ of a principal $G$-bundle $P\to M$ are invertible and form a group

\begin{equation}\label{eq:MorNGP}
	{\rm  Aut}({\CG^\bullet}(P))\simeq \CI_{G^\bullet\rtimes G}({\CP^\bullet}(P), G^\bullet).
\end{equation}
\end{cor}

Moreover, a groupoid morphism 
$\widehat f:\CP^{(1)}(P)\rightarrow {\CP^{(1)}}(P)$ induces a map $N^\bullet (\widehat f)\colon  \CP^\bullet (P) \fto  \CP^\bullet (P)$.   Also,  a principal bundle  morphism $f \in {\rm Aut}(P)$   induces  a groupoid morphism
\begin{eqnarray*}\widehat f\colon {\CP^{(1)}}(P)&\rightarrow& {\CP^{(1)}}(P)\\ (p,q)&\mapsto& (f(p),f(q)).\end{eqnarray*}
The corresponding nerve map $N^\bullet(\widehat f)\colon \CP^\bullet(P)\fto \CP^\bullet(P)$ reads:
$$N^{(k)}(\widehat f)(p_0,\dots,p_k)=(f(p_0),\dots,f(p_k)).$$

\begin{thm}\label{thm:embeddingsgaugetransf}Given a principal $G$-bundle $P\to M$, we have the following canonical injections of gauge group transformations: 
	\begin{equation}\label{eq:towermor} {\rm Aut} (P)  \overset{ \hat{}}{\hookrightarrow} {\rm Aut} ({\CP^{(1)}}(P)) \overset{ N^\bullet}{\hookrightarrow}   {\rm Aut} \left(  {\CP^\bullet}(P) \right) \simeq \CI_{G^\bullet\rtimes G} ( \CP^\bullet (P), G^\bullet\rtimes G), 
	\end{equation}    
\end{thm}  

\appendix

\section{Appendix: Morita equivalence}

In order to interpret $H\rtimes G\soutar G$-principal $2$-bundles in terms of principal bundle groupoids, we need the notion of Morita equivalence \cite{MK}, \cite{MM}, \cite{LTX}. {Here are three equivalent definitions.}

	{The following corresponds to a generalised isomorphism in the sense of Definition \ref{defn:generalisedhom}:}
	\begin{defn}\label{defn:Mor1}
		Let $G^{(1)}\soutar G$ and $H^{(1)}\soutar H$ be two Lie groupoids. A {\bf bitorsor} is a manifold $X$ and surjective submersions $G\overset{\rho}{\leftarrow} X\overset{\sigma}{\rightarrow} H$, such that \begin{enumerate}
			\item $X\overset{\rho}{\rightarrow} G$ is a  principal left $G^{(1)}$-bundle over $H$, { so $X/H^{(1)}\simeq G$,}
			\item $X\overset{\sigma}{\rightarrow} H$ is a principal right $H^{(1)}$-bundle over $G$, { so $X/G^{(1)}\simeq H$,}
			\item The $G^{(1)}$- and $H^{(1)}$-actions on $X$ commute, 
		\end{enumerate}
	{and  the following diagramm commutes:}	\[\begin{tikzcd}
		G^{(1)} \ar[d, shift left=.2em,"\bs"]\ar[d, shift right=.2em, swap,"\bt"] \arrow[r, bend left=30,swap, "\star"] & X \ar[dl, "\rho"] \ar[dr, swap, "\sigma"]& H^{(1)} \arrow[l, bend right=30,  "\star"] \ar[d, shift left=.2em,"\bs"]\ar[d, shift right=.2em, swap,"\bt"]  \\
		{G} & & { H}\\
		\end{tikzcd} \]
	\end{defn}

	\begin{defn}\label{defn:Mor2}
		Two Lie groupoids $G^{(1)}\soutar G$ and $H^{(1)}\soutar H$ are {\bf Morita equivalent} if there is a manifold $P$ with surjective submersions $G\overset{\rho}{\leftarrow} P\overset{\sigma}{\rightarrow} H$, such that $\rho^{-1}G^{(1)}\cong \sigma^{-1}H^{(1)}$ as Lie groupoids over $P$.
	\end{defn}
\begin{defn}
	Two Lie groupoids $G^{(1)}\soutar G$ and $H^{(1)}\soutar H$ are {\bf weak equivalent} if there is a groupoid manifold $P^{(1)}\soutar P$ with weak equivalences $G^{(1)}\overset{\rho}{\leftarrow} P^{(1)}\overset{\sigma}{\rightarrow} H^{(1)}$ (see Definition \ref{defn:Mor3} for the notion of weak equivalence of groupoids).
\end{defn}
\begin{ex}\label{ex:Y[1]} Let $\pi:Y\to M$ be a surjective submersion, and $M^{(1)}\soutar M$  be a Lie groupoid. The pull-back groupoid $\pi^{-1} M^{(1)} \soutar Y$ is Morita equivalent, and weak equivalent to $M^{(1)}\soutar M$, with $\rho={\rm Id}$ and $\sigma=\pi$. In particular, when $M^{(1)}=M$, then $\pi^{-1}M^{(1)}\simeq Y^{[2]}$ is Morita equivalent and weak equivalent to $M$.
\end{ex}

One easily checks that Morita equivalence and weak equivalence defines an equivalence relation. Another important property is its stability under generalised isomorphisms {(Definition \ref{defn:generalisedhom}).}

	\begin{lem}\label{lem:equivMor}
		Two Lie groupoids $G^{(1)}\soutar G$ and $H^{(1)}\soutar H$ are Morita equivalent if and only if they are weak equivalent or if and only if there exists a bitorsor.
	\end{lem}
	
	This was proven in \cite{MM} and \cite{LTX} along the lines briefly sketched here: A bitorsor is a Morita equivalence since the actions give an isomorphism of the pullback groupoids. A Morita equivalence gives a weak equivalence by means of the pullback groupoid in the middle. From a weak equivalence one can build a bitorsor. For further details, see \cite{G}.

\vfill \eject \noindent

\bibliographystyle{alpha}
\bibliography{MEfolbib}

\end{document}